\documentclass[11pt,a4paper]{amsart}

\usepackage[margin=2.5cm]{geometry}
\usepackage{amsmath}
\usepackage{amssymb}
\usepackage{amsthm}
\usepackage{tikz}
\usetikzlibrary{patterns}
\usepackage{float}
\usepackage{enumerate}
\usepackage{hyperref}
\usepackage{microtype}

\newtheorem{theorem}{Theorem}[section]
\newtheorem{proposition}[theorem]{Proposition}
\newtheorem{question}[theorem]{Question}
\newtheorem{corollary}[theorem]{Corollary}
\newtheorem{lemma}[theorem]{Lemma}
\theoremstyle{definition}
\newtheorem{definition}{Definition}[section]
\newtheorem{example}[definition]{Example}
\theoremstyle{remark}
\newtheorem{remark}{Remark}[section]

\newcommand{\va}{\mathbf{a}}
\newcommand{\vb}{\mathbf{b}}

\begin{document}

\title[Spanning tree enumeration via triangular rank-one perturbations]{Spanning tree enumeration via triangular rank-one perturbations of graph Laplacians}

\author{Christian Go}
\address{National University of Singapore, Department of Mathematics, Singapore}
\email{christian.go@nus.edu.sg}

\author{Khwa Zhong Xuan}
\address{Yale-NUS College, Division of Science, Singapore}
\email{zhongxuan@yale-nus.edu.sg}

\author{Xinyu Luo}
\address{Singapore, Singapore}
\email{luoxinyu06@gmail.com}

\author{Matthew T. Stamps}
\address{Ottawa, Ontario, Canada}
\email{matthewstamps@gmail.com}

\maketitle

\begin{abstract}
We present new short proofs of known spanning tree enumeration formulae for threshold and Ferrers graphs by showing that the Laplacian matrices of such graphs admit triangular rank-one perturbations.  We then characterize the set of graphs whose Laplacian matrices admit triangular rank-one perturbations as the class of special $2$-threshold graphs, introduced by Hung, Kloks, and Villaamil.  Our work introduces (1) a new characterization of special $2$-threshold graphs that generalizes the characterization of threshold graphs in terms of isolated and dominating vertices, and (2) a spanning tree enumeration formula for special $2$-threshold graphs that reduces to the aforementioned formulae for threshold and Ferrers graphs.  We consider both unweighted and weighted spanning tree enumeration.
\end{abstract}

\section{Introduction}

The motivation for this article originated with a linear algebraic technique for spanning tree enumeration introduced by Klee and Stamps~\cite{klee_stamps}. Kirchhoff's celebrated Matrix-Tree theorem~\cite{kirchhoff} asserts that the number of spanning trees $\tau(G)$ in a graph $G$ is encoded in any cofactor of its Laplacian matrix $L(G)$.  While the theorem ensures that $\tau(G)$ can be computed for a general graph $G$ in polynomial time, more efficient enumeration formulae exist for commonly studied classes of graphs, including complete~\cite{Aigner-Ziegler,Borchardt,cayley,EC1}, complete multipartite~\cite{Austin,Hartsfield-Werth,Lewis,Onodera,Scoins}, threshold~\cite{Chestnut-Fishkind,Hammer-Kelmans,Merris}, and Ferrers~\cite{E-VW} graphs. Klee and Stamps~\cite{klee_stamps} demonstrated that it can be easier to work with a rank-one perturbation of $L(G)$ rather than $L(G)$ itself.  They used their approach to give original straightforward proofs for each of the enumeration formulae listed above.  In this article, we present further simplified proofs of the spanning tree enumeration formulae for threshold and Ferrers graphs by showing that the vertices of such graphs can always be ordered so the associated Laplacian matrices admit triangular rank-one perturbations, which raises the following question:
 \begin{question}\label{question}
 Which graph Laplacians admit triangular rank-one perturbations? 
 \end{question}
We prove that the Laplacian matrix of a graph $G$ admits a triangular rank-one perturbation if and only if $G$ belongs to the class of \textit{special $2$-threshold} graphs introduced by Hung, Kloks, and Villaamil~\cite{HKV}. Our proof relies on a new characterization of special $2$-threshold graphs that generalizes the well known characterization of threshold graphs in terms of isolated and dominating vertices.  We also present a spanning tree enumeration formula for special $2$-threshold graphs that generalizes the aforementioned formulae for threshold and Ferrers graphs, revealing an unexpected connection between these seemingly unrelated classes of graphs. Our results extend naturally to weighted spanning tree enumerators, generalizing results of Martin and Reiner~\cite{Martin-Reiner} and Ehrenborg and van Willigenburg~\cite{E-VW}.  

The remaining sections of this article are organized as follows:  In Section~\ref{sec:prelims}, we set up some basic definitions and notation from graph theory, survey a number of spanning tree enumeration formulae in the literature, and review several linear algebraic techniques for spanning tree enumeration, which we use to present new short proofs of the spanning tree enumeration formulae for threshold and Ferrers graphs.  Section~\ref{sec:nuclear} is the main section of this article where we prove that a graph admits a triangular rank-one perturbation if and only if it is special $2$-threshold, and introduce a spanning tree enumeration formula for special $2$-threshold graphs that reduces to the known formulae for the special cases of threshold and Ferrers graphs.  This is also where we generalize the characterization of threshold graphs involving isolated and dominating vertices to special $2$-threshold graphs. We conclude the article with weighted versions of our main results in Section~\ref{sec:weighted}.   

\section{Background}\label{sec:prelims}

In this section, we go over the essential definitions, examples, and characterizations of the classes of graphs studied throughout this article, using standard notation wherever possible.  We also survey spanning tree enumeration formulae in the literature, review some linear algebraic techniques for spanning tree enumeration, and present new short proofs of established spanning tree enumeration formulae for threshold and Ferrers graphs.  

\subsection{Definitions and Examples}  

For the purposes of this article, a \textbf{graph} is a pair $G = (V,E)$ that consists of a finite set of \textbf{vertices}, $V = V(G)$, and a set of pairs of distinct vertices, $E = E(G)$, called \textbf{edges}. We will not consider graphs with loops (edges connecting a vertex to itself) or parallel edges (multiple edges between a given pair of vertices). 

\begin{example}
The graph $G = (V,E)$ with vertex and edge sets given by
$$V = \{1, 2, 3, 4, 5, 6\}\quad\text{and}\quad E = \{\{1,2\},\{1,4\},\{2,3\},\{2,5\},\{2,6\},\{4,5\},\{5,6\}\}$$  
is illustrated in Figure~\ref{fig:graphex}(a) alongside several other examples of graphs we shall examine in this section.  The vertices and edges of each graph are represented by dots and line segments, respectively.
\end{example}

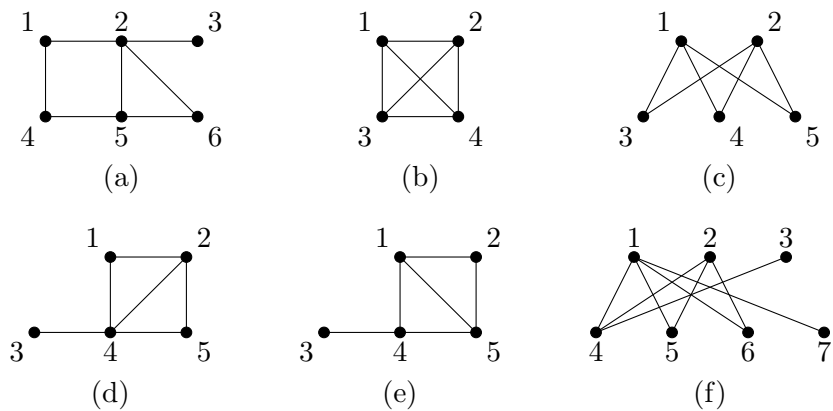
\begin{figure}[H]
    \centering
    \begin{tikzpicture}
    \filldraw (-1,1) circle (2pt);
    \node [above left] at (-1,1) {$1$};
    \filldraw (0,1) circle (2pt);
    \node [above] at (0,1) {$2$};
    \filldraw (1,1) circle (2pt);
    \node [above right] at (1,1) {$3$};
    \filldraw (-1,0) circle (2pt);
    \node [below left] at (-1,0) {$4$};
    \filldraw (0,0) circle (2pt);
    \node [below] at (0,0) {$5$};
    \filldraw (1,0) circle (2pt);
    \node [below right] at (1,0) {$6$};
    \draw (-1,1)--(0,1);
    \draw (-1,1)--(-1,0);
    \draw (0,1)--(1,0);
    \draw (0,1)--(0,0);
    \draw (0,1)--(1,1);
    \draw (-1,0)--(0,0);
    \draw (0,0)--(1,0);
    \node [below] at (0,-0.5) {(a)};
    \end{tikzpicture}
    \qquad\quad\;
    \begin{tikzpicture}
    \filldraw (0,0) circle (2pt);
    \node [below left] at (0,0) {$3$};
    \filldraw (0,1) circle (2pt);
    \node [above left] at (0,1) {$1$};
    \filldraw (1,1) circle (2pt);
    \node [above right] at (1,1) {$2$};
    \filldraw (1,0) circle (2pt);
    \node [below right] at (1,0) {$4$};
    \draw (0,0)--(0,1)--(1,1)--(1,0)--(0,0)--(1,1);
    \draw (0,1)--(1,0);
    \node [below] at (0.5,-0.5) {(b)};
    \end{tikzpicture}
    \qquad\quad\;
    \begin{tikzpicture}
    \filldraw (0,1) circle (2pt);
    \node [above left] at (0,1) {$1$};
    \filldraw (1,1) circle (2pt);
    \node [above right] at (1,1) {$2$};
    \filldraw (-0.5,0) circle (2pt);
    \node [below left] at (-0.5,0) {$3$};
    \filldraw (0.5,0) circle (2pt);
    \node [below right] at (0.5,0) {$4$};
    \filldraw (1.5,0) circle (2pt);
    \node [below right] at (1.5,0) {$5$};
    \draw (0,1)--(-0.5,0);
    \draw (0,1)--(0.5,0);
    \draw (0,1)--(1.5,0);
    \draw (1,1)--(-0.5,0);
    \draw (1,1)--(0.5,0);
    \draw (1,1)--(1.5,0);
    \node [below] at (0.5,-0.5) {(c)};
    \end{tikzpicture}
    
    \medskip
    
    \begin{tikzpicture}
        \filldraw (0,0) circle (2pt);
        \node [below left] at (0,0) {$3$};
        \filldraw (1,0) circle (2pt);
        \node [below] at (1,0) {$4$};
        \filldraw (1,1) circle (2pt);
        \node [above left] at (1,1) {$1$};
        \filldraw (2,1) circle (2pt);
        \node [above right] at (2,1) {$2$};
        \filldraw (2,0) circle (2pt);
        \node [below right] at (2,0) {$5$};
        \draw (0,0)--(1,0);
        \draw (1,0) -- (2,0) -- (2,1) -- (1,1) -- (1,0) -- (2,1);
        \node [below] at (1,-0.5) {(d)};
    \end{tikzpicture}
    \qquad
        \begin{tikzpicture}
        \filldraw (0,0) circle (2pt);
        \node [below left] at (0,0) {$3$};
        \filldraw (1,0) circle (2pt);
        \node [below] at (1,0) {$4$};
        \filldraw (1,1) circle (2pt);
        \node [above left] at (1,1) {$1$};
        \filldraw (2,1) circle (2pt);
        \node [above right] at (2,1) {$2$};
        \filldraw (2,0) circle (2pt);
        \node [below right] at (2,0) {$5$};
        \draw (0,0)--(1,0);
        \draw (1,1) -- (2,1) -- (2,0) -- (1,0) -- (1,1) -- (2,0);
        \node [below] at (1,-0.5) {(e)};
    \end{tikzpicture}
    \qquad
    \begin{tikzpicture}
        \filldraw (0,0) circle (2pt);
        \node[below] at (0,0) {$4$};
        \filldraw (1,0) circle (2pt);
        \node[below] at (1,0) {$5$};
        \filldraw (2,0) circle (2pt);
        \node[below] at (2,0) {$6$};
        \filldraw (3,0) circle (2pt);
        \node[below] at (3,0) {$7$};
        \filldraw (0.5,1) circle (2pt);
        \node[above] at (0.5,1) {$1$};
        \filldraw (1.5,1) circle (2pt);
        \node[above] at (1.5,1) {$2$};
        \filldraw (2.5,1) circle (2pt);
        \node[above] at (2.5,1) {$3$};
        \draw (0.5,1) -- (0,0);
        \draw (0.5,1) -- (1,0);
        \draw (0.5,1) -- (2,0);
        \draw (0.5,1) -- (3,0);
        \draw (1.5,1) -- (0,0);
        \draw (1.5,1) -- (1,0);
        \draw (1.5,1) -- (2,0);
        \draw (2.5,1) -- (0,0);
        \node [below] at (1.5,-0.5) {(f)};
    \end{tikzpicture}
    \caption{Examples of graphs.}
    \label{fig:graphex}
\end{figure}

Two vertices are \textbf{adjacent} if they form an edge.  A \textbf{path} from a vertex $u$ to a vertex $w$ is a sequence of edges $$\{v_0,v_1\}, \{v_1,v_2\}, \cdots, \{v_{k-1},v_k\}$$ such that $v_0 = u$, $v_k = w$, and no two $v_i$, except possibly $v_0$ and $v_k$, are equal.  A path from a vertex to itself is called a \textbf{cycle}. The \textbf{neighborhood} of a vertex is the set of all vertices adjacent to it. That is, for a graph $G = (V,E)$ and a vertex $v \in V$, the neighborhood of $v$ is the set $N(v) = \{ w \in V \ | \ \{v,w\} \in E\}$.  We use $N_W(v) = N(v) \cap W$ to denote the neighborhood of $v$ restricted to a subset $W \subseteq V$.  The \textbf{degree} of a vertex $v$, denoted by $\deg(v)$ is the cardinality of its neighborhood.  A graph $H$ is a \textbf{subgraph} of a graph $G$ if $V(H) \subseteq V(G)$ and $E(H) \subseteq E(G)$. The \textbf{induced subgraph} of $G$ on a subset $W \subseteq V(G)$ is the subgraph $G[W]$ with $$V(G[W]) = W \quad \text{and} \quad E(G[W]) = \{\{v,w\} \in E(G) \ | \ v,w \in W\}.$$ A subgraph $H$ of $G$ is \textbf{spanning} if $V(H) = V(G)$, \textbf{path-connected} if there exists a path in $E(H)$ between every pair of vertices in $V(H)$, and \textbf{acyclic} if $E(H)$ contains no cycles.  Subgraphs that satisfy all three of those properties are called \textbf{spanning trees}.   

\begin{example}
Four subgraphs of the graph in Figure~\ref{fig:graphex}(a) are illustrated in Figure~\ref{fig:subgraphs}: Subgraph (a) is path-connected and acyclic but not spanning, subgraph (b) is spanning and acyclic but not path-connected, and subgraph (c) is spanning and path-connected but not acyclic.  Subgraph (d) is the only spanning tree.

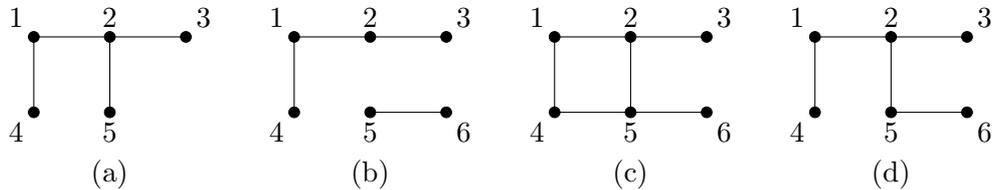
\begin{figure}[H]
    \centering
    \begin{tikzpicture}
    \filldraw (-1,1) circle (2pt);
    \node [above left] at (-1,1) {$1$};
    \filldraw (0,1) circle (2pt);
    \node [above] at (0,1) {$2$};
    \filldraw (1,1) circle (2pt);
    \node [above right] at (1,1) {$3$};
    \filldraw (-1,0) circle (2pt);
    \node [below left] at (-1,0) {$4$};
    \filldraw (0,0) circle (2pt);
    \node [below] at (0,0) {$5$};
    \draw (-1,1)--(0,1);
    \draw (-1,1)--(-1,0);
    \draw (0,1)--(0,0);
    \draw (0,1)--(1,1);
    \node [below] at (0,-0.5) {(a)};
    \end{tikzpicture}
    \quad
    \begin{tikzpicture}
    \filldraw (-1,1) circle (2pt);
    \node [above left] at (-1,1) {$1$};
    \filldraw (0,1) circle (2pt);
    \node [above] at (0,1) {$2$};
    \filldraw (1,1) circle (2pt);
    \node [above right] at (1,1) {$3$};
    \filldraw (-1,0) circle (2pt);
    \node [below left] at (-1,0) {$4$};
    \filldraw (0,0) circle (2pt);
    \node [below] at (0,0) {$5$};
    \filldraw (1,0) circle (2pt);
    \node [below right] at (1,0) {$6$};
    \draw (-1,1)--(0,1);
    \draw (-1,1)--(-1,0);
    \draw (0,1)--(1,1);
    \draw (0,0)--(1,0);
    \node [below] at (0,-0.5) {(b)};
    \end{tikzpicture}
    \quad
    \begin{tikzpicture}
    \filldraw (-1,1) circle (2pt);
    \node [above left] at (-1,1) {$1$};
    \filldraw (0,1) circle (2pt);
    \node [above] at (0,1) {$2$};
    \filldraw (1,1) circle (2pt);
    \node [above right] at (1,1) {$3$};
    \filldraw (-1,0) circle (2pt);
    \node [below left] at (-1,0) {$4$};
    \filldraw (0,0) circle (2pt);
    \node [below] at (0,0) {$5$};
    \filldraw (1,0) circle (2pt);
    \node [below right] at (1,0) {$6$};
    \draw (-1,1)--(0,1);
    \draw (-1,1)--(-1,0);
    \draw (0,1)--(0,0);
    \draw (0,1)--(1,1);
    \draw (-1,0)--(0,0);
    \draw (0,0)--(1,0);
    \node [below] at (0,-0.5) {(c)};
    \end{tikzpicture}
    \quad
    \begin{tikzpicture}
    \filldraw (-1,1) circle (2pt);
    \node [above left] at (-1,1) {$1$};
    \filldraw (0,1) circle (2pt);
    \node [above] at (0,1) {$2$};
    \filldraw (1,1) circle (2pt);
    \node [above right] at (1,1) {$3$};
    \filldraw (-1,0) circle (2pt);
    \node [below left] at (-1,0) {$4$};
    \filldraw (0,0) circle (2pt);
    \node [below] at (0,0) {$5$};
    \filldraw (1,0) circle (2pt);
    \node [below right] at (1,0) {$6$};
    \draw (-1,1)--(0,1);
    \draw (-1,1)--(-1,0);
    \draw (0,1)--(0,0);
    \draw (0,1)--(1,1);
    \draw (0,0)--(1,0);
    \node [below] at (0,-0.5) {(d)};
    \end{tikzpicture}
    \caption{An example and several non-examples of spanning trees.}
    \label{fig:subgraphs}
\end{figure}

The complete list of spanning trees of the graph in Figure~\ref{fig:graphex}(a) is illustrated in Figure~\ref{fig:spanningex}. 

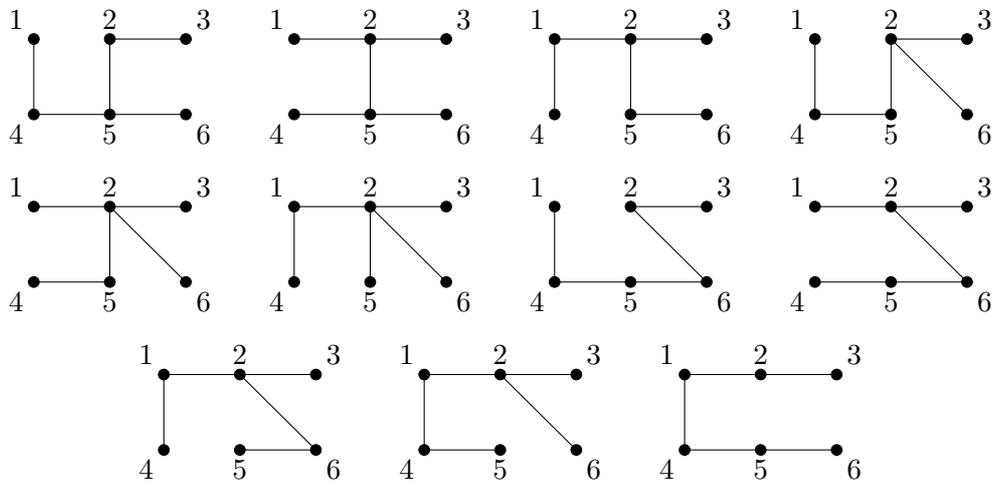
\begin{figure}[H]
    \centering
    \begin{tikzpicture}
    \filldraw (-1,1) circle (2pt);
    \node [above left] at (-1,1) {$1$};
    \filldraw (0,1) circle (2pt);
    \node [above] at (0,1) {$2$};
    \filldraw (1,1) circle (2pt);
    \node [above right] at (1,1) {$3$};
    \filldraw (-1,0) circle (2pt);
    \node [below left] at (-1,0) {$4$};
    \filldraw (0,0) circle (2pt);
    \node [below] at (0,0) {$5$};
    \filldraw (1,0) circle (2pt);
    \node [below right] at (1,0) {$6$};
    \draw (-1,1)--(-1,0);
    \draw (0,1)--(0,0);
    \draw (0,1)--(1,1);
    \draw (-1,0)--(0,0);
    \draw (0,0)--(1,0);
    \end{tikzpicture}
    \quad
        \begin{tikzpicture}
    \filldraw (-1,1) circle (2pt);
    \node [above left] at (-1,1) {$1$};
    \filldraw (0,1) circle (2pt);
    \node [above] at (0,1) {$2$};
    \filldraw (1,1) circle (2pt);
    \node [above right] at (1,1) {$3$};
    \filldraw (-1,0) circle (2pt);
    \node [below left] at (-1,0) {$4$};
    \filldraw (0,0) circle (2pt);
    \node [below] at (0,0) {$5$};
    \filldraw (1,0) circle (2pt);
    \node [below right] at (1,0) {$6$};
    \draw (-1,1)--(0,1);
    \draw (0,1)--(0,0);
    \draw (0,1)--(1,1);
    \draw (-1,0)--(0,0);
    \draw (0,0)--(1,0);
    \end{tikzpicture}
    \quad
        \begin{tikzpicture}
    \filldraw (-1,1) circle (2pt);
    \node [above left] at (-1,1) {$1$};
    \filldraw (0,1) circle (2pt);
    \node [above] at (0,1) {$2$};
    \filldraw (1,1) circle (2pt);
    \node [above right] at (1,1) {$3$};
    \filldraw (-1,0) circle (2pt);
    \node [below left] at (-1,0) {$4$};
    \filldraw (0,0) circle (2pt);
    \node [below] at (0,0) {$5$};
    \filldraw (1,0) circle (2pt);
    \node [below right] at (1,0) {$6$};
    \draw (-1,1)--(0,1);
    \draw (-1,1)--(-1,0);
    \draw (0,1)--(0,0);
    \draw (0,1)--(1,1);
    \draw (0,0)--(1,0);
    \end{tikzpicture}
    \quad
        \begin{tikzpicture}
    \filldraw (-1,1) circle (2pt);
    \node [above left] at (-1,1) {$1$};
    \filldraw (0,1) circle (2pt);
    \node [above] at (0,1) {$2$};
    \filldraw (1,1) circle (2pt);
    \node [above right] at (1,1) {$3$};
    \filldraw (-1,0) circle (2pt);
    \node [below left] at (-1,0) {$4$};
    \filldraw (0,0) circle (2pt);
    \node [below] at (0,0) {$5$};
    \filldraw (1,0) circle (2pt);
    \node [below right] at (1,0) {$6$};
    \draw (-1,1)--(-1,0);
    \draw (0,1)--(1,0);
    \draw (0,1)--(0,0);
    \draw (0,1)--(1,1);
    \draw (-1,0)--(0,0);
    \end{tikzpicture}
    
    \medskip
    
    \begin{tikzpicture}
    \filldraw (-1,1) circle (2pt);
    \node [above left] at (-1,1) {$1$};
    \filldraw (0,1) circle (2pt);
    \node [above] at (0,1) {$2$};
    \filldraw (1,1) circle (2pt);
    \node [above right] at (1,1) {$3$};
    \filldraw (-1,0) circle (2pt);
    \node [below left] at (-1,0) {$4$};
    \filldraw (0,0) circle (2pt);
    \node [below] at (0,0) {$5$};
    \filldraw (1,0) circle (2pt);
    \node [below right] at (1,0) {$6$};
    \draw (-1,1)--(0,1);
    \draw (0,1)--(1,0);
    \draw (0,1)--(0,0);
    \draw (0,1)--(1,1);
    \draw (-1,0)--(0,0);
    \end{tikzpicture}
    \quad
    \begin{tikzpicture}
    \filldraw (-1,1) circle (2pt);
    \node [above left] at (-1,1) {$1$};
    \filldraw (0,1) circle (2pt);
    \node [above] at (0,1) {$2$};
    \filldraw (1,1) circle (2pt);
    \node [above right] at (1,1) {$3$};
    \filldraw (-1,0) circle (2pt);
    \node [below left] at (-1,0) {$4$};
    \filldraw (0,0) circle (2pt);
    \node [below] at (0,0) {$5$};
    \filldraw (1,0) circle (2pt);
    \node [below right] at (1,0) {$6$};
    \draw (-1,1)--(0,1);
    \draw (-1,1)--(-1,0);
    \draw (0,1)--(1,0);
    \draw (0,1)--(0,0);
    \draw (0,1)--(1,1);
    \end{tikzpicture}
    \quad
        \begin{tikzpicture}
    \filldraw (-1,1) circle (2pt);
    \node [above left] at (-1,1) {$1$};
    \filldraw (0,1) circle (2pt);
    \node [above] at (0,1) {$2$};
    \filldraw (1,1) circle (2pt);
    \node [above right] at (1,1) {$3$};
    \filldraw (-1,0) circle (2pt);
    \node [below left] at (-1,0) {$4$};
    \filldraw (0,0) circle (2pt);
    \node [below] at (0,0) {$5$};
    \filldraw (1,0) circle (2pt);
    \node [below right] at (1,0) {$6$};
    \draw (-1,1)--(-1,0);
    \draw (0,1)--(1,0);
    \draw (0,1)--(1,1);
    \draw (-1,0)--(0,0);
    \draw (0,0)--(1,0);
    \end{tikzpicture}
    \quad
        \begin{tikzpicture}
    \filldraw (-1,1) circle (2pt);
    \node [above left] at (-1,1) {$1$};
    \filldraw (0,1) circle (2pt);
    \node [above] at (0,1) {$2$};
    \filldraw (1,1) circle (2pt);
    \node [above right] at (1,1) {$3$};
    \filldraw (-1,0) circle (2pt);
    \node [below left] at (-1,0) {$4$};
    \filldraw (0,0) circle (2pt);
    \node [below] at (0,0) {$5$};
    \filldraw (1,0) circle (2pt);
    \node [below right] at (1,0) {$6$};
    \draw (-1,1)--(0,1);
    \draw (0,1)--(1,0);
    \draw (0,1)--(1,1);
    \draw (-1,0)--(0,0);
    \draw (0,0)--(1,0);
    \end{tikzpicture}
    
    \medskip
    
    \begin{tikzpicture}
    \filldraw (-1,1) circle (2pt);
    \node [above left] at (-1,1) {$1$};
    \filldraw (0,1) circle (2pt);
    \node [above] at (0,1) {$2$};
    \filldraw (1,1) circle (2pt);
    \node [above right] at (1,1) {$3$};
    \filldraw (-1,0) circle (2pt);
    \node [below left] at (-1,0) {$4$};
    \filldraw (0,0) circle (2pt);
    \node [below] at (0,0) {$5$};
    \filldraw (1,0) circle (2pt);
    \node [below right] at (1,0) {$6$};
    \draw (-1,1)--(0,1);
    \draw (-1,1)--(-1,0);
    \draw (0,1)--(1,0);
    \draw (0,1)--(1,1);
    \draw (0,0)--(1,0);
    \end{tikzpicture}
    \quad 
    \begin{tikzpicture}
    \filldraw (-1,1) circle (2pt);
    \node [above left] at (-1,1) {$1$};
    \filldraw (0,1) circle (2pt);
    \node [above] at (0,1) {$2$};
    \filldraw (1,1) circle (2pt);
    \node [above right] at (1,1) {$3$};
    \filldraw (-1,0) circle (2pt);
    \node [below left] at (-1,0) {$4$};
    \filldraw (0,0) circle (2pt);
    \node [below] at (0,0) {$5$};
    \filldraw (1,0) circle (2pt);
    \node [below right] at (1,0) {$6$};
    \draw (-1,1)--(0,1);
    \draw (-1,1)--(-1,0);
    \draw (0,1)--(1,0);
    \draw (0,1)--(1,1);
    \draw (-1,0)--(0,0);
    \end{tikzpicture}
    \quad 
    \begin{tikzpicture}
    \filldraw (-1,1) circle (2pt);
    \node [above left] at (-1,1) {$1$};
    \filldraw (0,1) circle (2pt);
    \node [above] at (0,1) {$2$};
    \filldraw (1,1) circle (2pt);
    \node [above right] at (1,1) {$3$};
    \filldraw (-1,0) circle (2pt);
    \node [below left] at (-1,0) {$4$};
    \filldraw (0,0) circle (2pt);
    \node [below] at (0,0) {$5$};
    \filldraw (1,0) circle (2pt);
    \node [below right] at (1,0) {$6$};
    \draw (-1,1)--(0,1);
    \draw (-1,1)--(-1,0);
    \draw (0,1)--(1,1);
    \draw (-1,0)--(0,0);
    \draw (0,0)--(1,0);
    \end{tikzpicture}
    \caption{All spanning trees of the graph $G$ in Figure~\ref{fig:graphex}(a).}
    \label{fig:spanningex}
\end{figure}
\end{example}


\subsubsection{Complete Graphs \& Generalizations} The first class of graphs we consider in this article are the complete graphs.  A graph is \textbf{complete} if every pair of distinct vertices forms an edge.  A complete graph on $n$ vertices is customarily denoted by $K_n$.  The graph illustrated in Figure~\ref{fig:graphex}(b) is a complete graph on four vertices, $K_4$.  A natural generalization of complete graphs are complete multipartite graphs. A graph $G = (V,E)$ is \textbf{complete multipartite} if $V$ can be partitioned into (disjoint) subsets $V_1, \ldots, V_k$ such that two vertices form an edge if and only if they belong to different subsets.  Such a graph is customarily denoted by $K_{n_1,\ldots,n_k}$ where $n_i = |V_i|$ for $1 \leq i \leq k$.  The graph illustrated in Figure~\ref{fig:graphex}(c) is a complete bipartite graph, $K_{2,3}$, with $V_1 = \{1,2\}$ and $V_2 = \{3,4,5\}$.

\subsubsection{Threshold Graphs \& Generalizations} To introduce the next class of graphs, we must first define isolated and dominating vertices.  A vertex is \textbf{isolated} in a graph $G$ if it is not contained in any edges of $G$ and \textbf{dominating} in $G$ if it forms an edge with every other vertex in $G$. A graph $G$ is \textbf{threshold} if every induced subgraph contains an isolated or dominating vertex.  Threshold graphs were first introduced by Chv\'atal and Hammer in \cite{chvatal_hammer_1977}, and they have many equivalent characterizations~\cite{Mahadev-Peled}.  

One of the best known characterizations is that every threshold graph can be constructed from an initial vertex by repeatedly appending either an isolated vertex or a dominating vertex. To be more precise, given an ordering $v_1, \ldots, v_n$ of the vertices of a graph $G$, the \textbf{lower neighborhood} of the vertex $v_i$, denoted by $N^< (v_i)$, is the set of all vertices adjacent to $v_i$ with an index lower than $i,$ that is, $N^< (v_i) = N(v_i) \cap \{v_1,\dots,v_{i-1}\}.$

\begin{theorem}[\cite{chvatal_hammer_1977}]\label{thm:threshold-order}
A graph $G$ on $n$ vertices is threshold if and only if its vertices can be ordered $v_1, \ldots, v_n$ such that $$N^<(v_i) = \emptyset \quad \text{or} \quad N^<(v_i) = \{v_1,\ldots,v_{i-1}\}$$ for every $i \in [n]$.
\end{theorem}

We call such an ordering of the vertices of a threshold graph a \textbf{construction order}.

\begin{example} The graph in Figure~\ref{fig:graphex}(d) has a construction order $v_1 = 1$, $v_2 = 5$, $v_3 = 2$, $v_4 = 3$, and $v_5 = 4$.  Indeed, $v_1$ is an initial vertex, $v_2$ and $v_4$ are isolated vertices, and $v_3$ and $v_5$ are dominating vertices, as illustrated below.

\begin{figure}[H]
    \centering
    \begin{tikzpicture}
        \filldraw[opacity=0.25] (0,0) circle (2pt);
        \node [opacity=0.25,below left] at (0,0) {$v_4$};
        \filldraw[opacity=0.25]  (1,0) circle (2pt);
        \node [opacity=0.25,below] at (1,0) {$v_5$};
        \filldraw  (1,1) circle (2pt);
        \node [above left] at (1,1) {$v_1$};
        \filldraw[opacity=0.25]  (2,1) circle (2pt);
        \node [opacity=0.25,above right] at (2,1) {$v_3$};
        \filldraw  (2,0) circle (2pt);
        \node [below right] at (2,0) {$v_2$};
        \draw[opacity=0.25] (0,0)--(1,0);
        \draw[opacity=0.25] (1,0) -- (2,0) -- (2,1) -- (1,1) -- (1,0) -- (2,1);
    \end{tikzpicture}
    \quad
    \begin{tikzpicture}
        \filldraw[opacity=0.25] (0,0) circle (2pt);
        \node [opacity=0.25,below left] at (0,0) {$v_4$};
        \filldraw[opacity=0.25] (1,0) circle (2pt);
        \node [opacity=0.25,below] at (1,0) {$v_5$};
        \filldraw (1,1) circle (2pt);
        \node [above left] at (1,1) {$v_1$};
        \filldraw (2,1) circle (2pt);
        \node [above right] at (2,1) {$v_3$};
        \filldraw (2,0) circle (2pt);
        \node [below right] at (2,0) {$v_2$};
        \draw[opacity=0.25] (0,0)--(1,0);
        \draw[opacity=0.25] (1,0) -- (2,0) -- (2,1) -- (1,1) -- (1,0) -- (2,1);
        \draw (1,1)--(2,1)--(2,0);
    \end{tikzpicture}
    \quad
    \begin{tikzpicture}
        \filldraw (0,0) circle (2pt);
        \node [below left] at (0,0) {$v_4$};
        \filldraw[opacity=0.25] (1,0) circle (2pt);
        \node [opacity=0.25,below] at (1,0) {$v_5$};
        \filldraw (1,1) circle (2pt);
        \node [above left] at (1,1) {$v_1$};
        \filldraw (2,1) circle (2pt);
        \node [above right] at (2,1) {$v_3$};
        \filldraw (2,0) circle (2pt);
        \node [below right] at (2,0) {$v_2$};
        \draw[opacity=0.25] (0,0)--(1,0);
        \draw[opacity=0.25] (1,0) -- (2,0) -- (2,1) -- (1,1) -- (1,0) -- (2,1);
        \draw (1,1)--(2,1)--(2,0);
    \end{tikzpicture}
    \quad
    \begin{tikzpicture}
        \filldraw (0,0) circle (2pt);
        \node [below left] at (0,0) {$v_4$};
        \filldraw (1,0) circle (2pt);
        \node [below] at (1,0) {$v_5$};
        \filldraw (1,1) circle (2pt);
        \node [above left] at (1,1) {$v_1$};
        \filldraw (2,1) circle (2pt);
        \node [above right] at (2,1) {$v_3$};
        \filldraw (2,0) circle (2pt);
        \node [below right] at (2,0) {$v_2$};
        \draw (0,0)--(1,0);
        \draw (1,0) -- (2,0) -- (2,1) -- (1,1) -- (1,0) -- (2,1);
    \end{tikzpicture}
    \caption{A construction order for the vertices of a threshold graph.}
    \label{fig:threshold}
\end{figure}
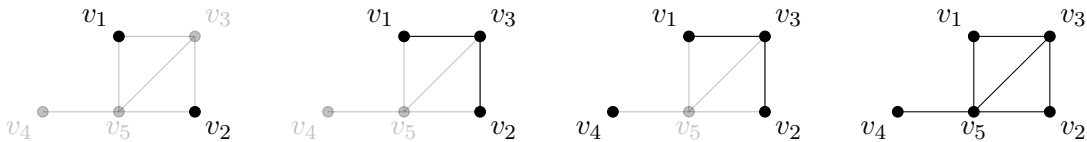
\end{example}

Threshold graphs can also be characterized in terms of forbidden induced subgraphs.

\begin{theorem}[\cite{chvatal_hammer_1977}]\label{thm:thresholdforbidden}
    A graph is threshold if and only if it does not contain an induced subgraph isomorphic to $2K_2$, $P_4$, or $C_4$.
\end{theorem}

\begin{figure}[H]
    \centering
        \begin{tikzpicture}
        \node[label=below:$2K_2$] at (0.5,0) {};
        \filldraw (0,0) circle (2pt);
        \filldraw (0,1) circle (2pt);
        \filldraw (1,1) circle (2pt);
        \filldraw (1,0) circle (2pt);
        \draw (1,1)--(1,0);
        \draw (0,0)--(0,1);
    \end{tikzpicture}
    \qquad \qquad
    \begin{tikzpicture}
    \node[label=below:$P_4$] at (0.5,0) {};
        \filldraw (0,0) circle (2pt);
        \filldraw (0,1) circle (2pt);
        \filldraw (1,1) circle (2pt);
        \filldraw (1,0) circle (2pt);
        \draw (0,0)--(0,1)--(1,1)--(1,0);
    \end{tikzpicture}
    \qquad \qquad
    \begin{tikzpicture}
    \node[label=below:$C_4$] at (0.5,0) {};
        \filldraw (0,0) circle (2pt);
        \filldraw (0,1) circle (2pt);
        \filldraw (1,1) circle (2pt);
        \filldraw (1,0) circle (2pt);
        \draw (0,0)--(1,0)--(1,1)--(0,1)--(0,0);
    \end{tikzpicture}
   \caption{The forbidden induced subgraphs of threshold graphs}
    \label{fig:2K2P4C4}
\end{figure}
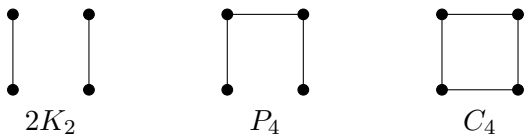

With Theorem~\ref{thm:thresholdforbidden}, one can check that the graph illustrated in Figure~\ref{fig:graphex}(d) is threshold since the deletion of any vertex (or subset of vertices) does not yield a $2K_2$, $P_4$, or $C_4$. The graph illustrated in Figure~\ref{fig:graphex}(e), on the other hand, is \emph{not} threshold since the induced subgraph on $W = \{1,2,3,4\}$ is isomorphic to $P_4$.

Hung, Kloks, and Villaamil extended the definition of threshold graphs to $k$-threshold graphs in \cite{HKV}. A graph $G = (V,E)$ is \textbf{\emph{k}-threshold} for some positive integer $k$ if $V$ can be partitioned into $k$ disjoint, possibly empty, subsets $V_1, \ldots, V_k$ such that for any $W \subseteq V$, the induced subgraph $G[W]$ contains a vertex $v$ such that $N_W(v) = \emptyset$ or $N_W(v) = (W \setminus \{v\}) \cap V_i$ for some $i \in [k]$.  In the case that $k=2$, a graph $G$ is \textbf{special \emph{2}-threshold} if there exists a bipartition $V = V_1 \sqcup V_2$  such that for every subset $W \subseteq V$, the induced subgraph $G[W]$ contains a vertex $v$ such that $N_W(v) = \emptyset$ or $N_W(v) = (W \setminus \{v\}) \cap V_1$.  

\begin{remark}\label{remark:threshold-special}
Every threshold graph $G$ is special $2$-threshold with $V_1 = V(G)$ and $V_2 = \emptyset$.
\end{remark}

Note that the bipartition in the definition of special $2$-threshold is not necessarily unique.  For instances when it is convenient to specify a particular bipartition of a special $2$-threshold graph, we introduce the following definition: A graph $G = (V,E)$ is \textbf{\emph{U}-threshold} for $U \subseteq V$ if for every subset $W \subseteq V$, the induced subgraph $G[W]$ contains a vertex $v$ such that $N_W(v) = \emptyset$ or $N_W(v) = (W \setminus \{v\}) \cap U$. 

\begin{remark}\label{rem:U-threshold}
A graph $G = (V,E)$ is special $2$-threshold if and only if it is $U$-threshold for some $U \subseteq V$ by setting $V_1 = U$ and $V_2 = V \setminus U$ in the respective definitions.
\end{remark}

\newpage

As with threshold graphs, special $2$-threshold graphs can be characterized in terms of forbidden induced subgraphs.

\begin{theorem}[\cite{HKV}]
  \label{conj:forbidden}
  A graph is special $2$-threshold if and only if it does not contain an
  induced subgraph isomorphic to a $2K_2$, $C_5$, House, Gem, Net, Diamond+2P, $W_4$+P, or Octahedron.
\end{theorem}

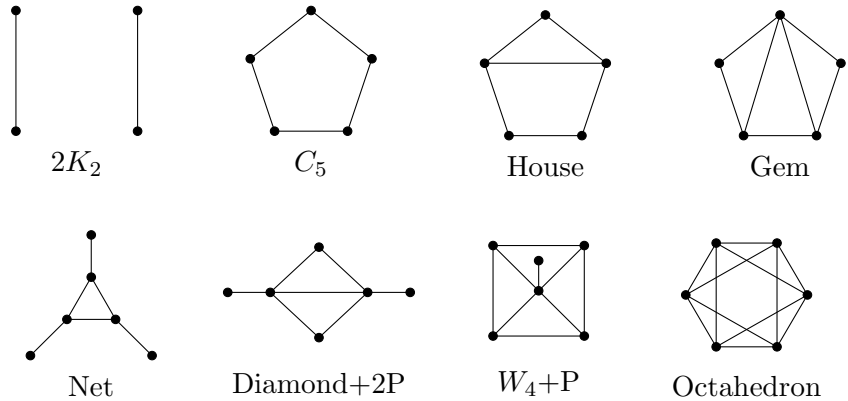
\begin{figure}[H]
  \centering
  \begin{tikzpicture}[scale=0.8]
    \node[label=below:$2K_2$] at (1,0) {};
    \filldraw[] (0,0) circle (2pt);
    \filldraw[] (2,0) circle (2pt);
    \filldraw[] (0,2) circle (2pt);
    \filldraw[] (2,2) circle (2pt);
    \draw[] (0,0) -- (0,2);
    \draw[] (2,0) -- (2,2);
  \end{tikzpicture}
  \hspace{1.1cm}
  \begin{tikzpicture}[scale=0.8]
    \node[label=below:$C_5$] at (1,0) {};
    \filldraw[] (1,2) circle (2pt);
    \filldraw[] (0,1.2) circle (2pt);
    \filldraw[] (2,1.2) circle (2pt);
    \filldraw[] (0.4,0) circle (2pt);
    \filldraw[] (1.6,0) circle (2pt);
    \draw[] (1,2) -- (0, 1.2) -- (0.4, 0) -- (1.6, 0) -- (2, 1.2) -- (1,2);
  \end{tikzpicture}
  \hspace{1.1cm}
  \begin{tikzpicture}[scale=0.8]
    \node[label=below:House] at (1,0) {};
    \filldraw[] (1,2) circle (2pt);
    \filldraw[] (0,1.2) circle (2pt);
    \filldraw[] (2,1.2) circle (2pt);
    \filldraw[] (0.4,0) circle (2pt);
    \filldraw[] (1.6,0) circle (2pt);
    \draw[] (1,2) -- (0, 1.2) -- (0.4, 0) -- (1.6, 0) -- (2, 1.2) -- (1,2);
    \draw[] (0, 1.2) -- (2, 1.2);
  \end{tikzpicture}
  \hspace{1.1cm}
  \begin{tikzpicture}[scale=0.8]
    \node[label=below:Gem] at (1,0) {};
    \filldraw[] (1,2) circle (2pt);
    \filldraw[] (0,1.2) circle (2pt);
    \filldraw[] (2,1.2) circle (2pt);
    \filldraw[] (0.4,0) circle (2pt);
    \filldraw[] (1.6,0) circle (2pt);
    \draw[] (1,2) -- (0, 1.2) -- (0.4, 0) -- (1.6, 0) -- (2, 1.2) -- (1,2);
    \draw[] (0.4, 0) -- (1,2) -- (1.6, 0);
  \end{tikzpicture}\\
  \vspace{0.5cm}
  \begin{tikzpicture}[scale=0.8]
    \node[label=below:Net] at (1,0) {};
    \filldraw[] (1,2) circle (2pt);
    \filldraw[] (1,1.3) circle (2pt);
    \filldraw[] (0.6, 0.6) circle (2pt);
    \filldraw[] (0,0) circle (2pt);
    \filldraw[] (2,0) circle (2pt);
    \filldraw[] (1.4,0.6) circle (2pt);
    \draw[] (1,2) -- (1,1.3) -- (0.6, 0.6);
    \draw[] (0,0) -- (0.6, 0.6) -- (1.4, 0.6);
    \draw[] (2,0) -- (1.4, 0.6) -- (1, 1.3);
  \end{tikzpicture}
  \hspace{0.6cm}
    \begin{tikzpicture}[scale=0.8]
    \node[label=below: Diamond+2P] at (1.5,0) {};
    \filldraw[] (0,1) circle (2pt);
    \filldraw[] (3,1) circle (2pt);
    \filldraw[] (0.7,1) circle (2pt);
    \filldraw[] (2.3,1) circle (2pt);
    \filldraw[] (1.5,0.25) circle (2pt);
    \filldraw[] (1.5,1.75) circle (2pt);
    \draw[] (0, 1) -- (3, 1);
    \draw[] (0.7,1) -- (1.5, 0.25) -- (2.3, 1) -- (1.5, 1.75) -- (0.7, 1);
  \end{tikzpicture}
  \hspace{0.5cm}
  \begin{tikzpicture}[scale=0.8]
    \node[label=below:{$W_4$+P}] at (1,0) {};
    \filldraw[white] (2,2) circle (2pt);
    \filldraw[white] (0,0) circle (2pt);
    \filldraw[] (0.25, 1.75) circle (2pt);
    \filldraw[] (0.25, 0.25) circle (2pt);
    \filldraw[] (1.75, 1.75) circle (2pt);
    \filldraw[] (1.75, 0.25) circle (2pt);
    \filldraw[] (1, 1) circle (2pt);
    \filldraw[] (1, 1.5) circle (2pt);
    \draw[] (0.25,0.25) rectangle (1.75,1.75);
    \draw[] (0.25,0.25) -- (1, 1) -- (1.75, 1.75);
    \draw[] (0.25, 1.75) -- (1, 1) -- (1.75, 0.25);
    \draw[] (1, 1) -- (1,1.5) ;
  \end{tikzpicture}
  \hspace{0.5cm}
  \begin{tikzpicture}[scale=0.8]
    \node[label=below: Octahedron] at (1,0) {};
    \filldraw[] (0,1) circle (2pt);
    \filldraw[] (2,1) circle (2pt);
    \filldraw[] (0.5, 1.86) circle (2pt);
    \filldraw[] (1.5, 1.86) circle (2pt);
    \filldraw[] (0.5, 0.14) circle (2pt);
    \filldraw[] (1.5, 0.14) circle (2pt);
    \draw[] (0.5, 1.86) rectangle (1.5, 0.14);
    \draw[] (0,1) -- (1.5, 0.14) -- (2,1) -- (0.5, 1.86) -- (0,1);
    \draw[] (0,1) -- (0.5, 0.14) -- (2,1) -- (1.5, 1.86) -- (0,1);
  \end{tikzpicture}
  \caption{The forbidden induced subgraphs for special $2$-threshold graphs.}
  \label{fig:forbidden}
\end{figure}

With Theorem~\ref{conj:forbidden}, one can verify that while the graph illustrated in Figure~\ref{fig:graphex}(e) is not a threshold graph, it is a special $2$-threshold graph since the removal of any subset of vertices does not yield one of the eight forbidden induced subgraphs illustrated in Figure~\ref{fig:forbidden}.  

\subsubsection{Ferrers Graphs} The final class of graphs we consider in this article are Ferrers graphs, which are bipartite graphs associated with Ferrers diagrams of integer partitions, first studied by Ehrenborg and van Willigenburg in \cite{E-VW}. A \textbf{partition} of a positive integer $n$ is a weakly decreasing sequence of positive integers $\lambda = (\lambda_1 , \dots , \lambda _m)$ that sum to $n$. The \textbf{Ferrers diagram} associated to $\lambda$ is an arrangement of $n$ square boxes into $m$ left-justified rows whose $i$th row consists of $\lambda_i$ boxes and the \textbf{Ferrers graph} associated to $\lambda$ is the graph on $m+\lambda_1$ vertices indexed by the rows and columns of its Ferrers diagram whose edges are the row-column pairs that contain boxes.

\begin{example}
The Ferrers diagram and Ferrers graph of the partition $(3,2,2,1)$ of $8$ are illustrated in Figure~\ref{fig:ferrers}.

\begin{figure}[H]
    \centering
    \begin{tikzpicture}[scale=0.60]
    \draw (0,0) grid (1,4);
    \draw (1,1) grid (2,4);
    \draw (2,3) grid (3,4);
    \node[left] at (0,0.5) {$r_4$};
    \node[left] at (0,1.5) {$r_3$};
    \node[left] at (0,2.5) {$r_2$};
    \node[left] at (0,3.5) {$r_1$};
    \node[above] at (0.5,4) {$c_1$};
    \node[above] at (1.5,4) {$c_2$};
    \node[above] at (2.5,4) {$c_3$};
    \draw[line width=2pt] (0,0) -- (1, 0) -- (1, 1) -- (2, 1) -- (2, 3) -- (3, 3) -- (3, 4);
    \end{tikzpicture}
    \qquad \qquad
    \begin{tikzpicture}
        \filldraw (0,0) circle (2pt);
        \node[below] at (0,0) {$r_1$};
        \filldraw (1,0) circle (2pt);
        \node[below] at (1,0) {$r_2$};
        \filldraw (2,0) circle (2pt);
        \node[below] at (2,0) {$r_3$};
        \filldraw (3,0) circle (2pt);
        \node[below] at (3,0) {$r_4$};
        \filldraw (0.5,1.5) circle (2pt);
        \node[above] at (0.5,1.5) {$c_1$};
        \filldraw (1.5,1.5) circle (2pt);
        \node[above] at (1.5,1.5) {$c_2$};
        \filldraw (2.5,1.5) circle (2pt);
        \node[above] at (2.5,1.5) {$c_3$};
        \draw (0.5,1.5) -- (0,0);
        \draw (0.5,1.5) -- (1,0);
        \draw (0.5,1.5) -- (2,0);
        \draw (0.5,1.5) -- (3,0);
        \draw (1.5,1.5) -- (0,0);
        \draw (1.5,1.5) -- (1,0);
        \draw (1.5,1.5) -- (2,0);
        \draw (2.5,1.5) -- (0,0);
    \end{tikzpicture}
    
    \caption{The Ferrers diagram (left) and Ferrers graph (right) of the partition $(3,2,2,1)$ of $8$.}
    \label{fig:ferrers}
\end{figure}
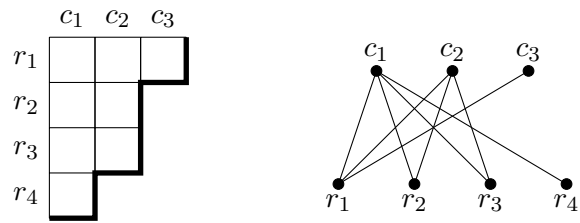
\end{example}

A closer look at the graph in Figure~\ref{fig:graphex}(f) reveals that it is identical to the Ferrers graph in Figure~\ref{fig:ferrers}.  In general, a bipartite graph $G=(V,E)$ is Ferrers if $V$ can be partitioned as $R\sqcup C$ with $R = \{r_1,\dots,r_n\}$ and $C=\{c_1,\dots,c_m\}$ such that
\begin{enumerate}
    \item $\{r_1 ,c_m\} \in E$ and $\{r_n, c_1\} \in E$; and
    \item if $\{r_k, c_\ell\}\in E$, then $\{r_i, c_j\} \in E$ for all $1\leq i \leq k, 1 \leq j \leq \ell.$
\end{enumerate}
Setting $R_k = \Big\{r_i \ \Big| \ \underset{c_j \in N(r_i)}{\max} j = k \Big\}$ for $1 \leq k \leq m$, we define the \textbf{traversal} of $G$ as the refinement of the ordering $$c_1 \leq R_1 \leq c_2 \leq R_2 \leq \cdots \leq c_m \leq R_m$$ with the elements of $R_k$ arranged in decreasing order of index.  For instance, the traversal for the Ferrers graph in Figure~\ref{fig:ferrers} is $c_1, r_4, c_2, r_3, r_2, c_3, r_1$ since $R_1 = \{r_4\}$, $R_2 = \{r_2,r_3\}$, and $R_3 = \{r_1\}$.  The traversal of a Ferrers graph can be read off easily from its associated Ferrers diagram by recording the vertex corresponding to each step along the bottom boundary of the diagram beginning at the bottom left corner and terminating at the top right corner, as highlighted with bold line segments in Figure~\ref{fig:ferrers}.

As with threshold and special $2$-threshold graphs, Ferrers graphs have a forbidden subgraph characterization.   

\begin{theorem}[\cite{liu}]\label{thm:Ferrersforbidden}
    A connected bipartite graph is Ferrers if and only if it does not contain an induced matching, i.e., an induced subgraph isomorphic to $2K_2$.
\end{theorem}

Since every forbidden induced subgraph of a special $2$-threshold graph apart from $2K_2$ contains an odd cycle, it follows directly from Theorems~\ref{conj:forbidden} and \ref{thm:Ferrersforbidden} that every Ferrers graph is special $2$-threshold.  That result can also be proved directly from the definitions. 

\begin{proposition}
  \label{prop:ferrersqt}
  Every Ferrers graph is special $2$-threshold.
\end{proposition}

\begin{proof}
Let $G = (V,E)$ be a Ferrers graph with $V = R \sqcup C$ and let $W \subseteq V$.  If $W \cap C = \emptyset$, then $W \subseteq R$, which means every vertex of $G[W]$ is isolated.  Otherwise, let $c$ be the highest indexed vertex in $W \cap C$.  If $c$ is not isolated in $G[W]$, then there exists an $r \in W \cap R$ that is adjacent to $c$.  Since $G$ is Ferrers, $r$ must be adjacent to every other vertex in $W \cap C$.  Thus, $G$ is a special $2$-threshold graph with $V_1 = C$ and $V_2 = R$. 
\end{proof}

We have seen that threshold and Ferrers graphs are both subclasses of special $2$-threshold graphs.  Note that there are special $2$-threshold graphs that are neither threshold nor Ferrers.  For instance, the graph in Figure~\ref{fig:graphex}(e), which we already established to be special $2$-threshold but not threshold, cannot be Ferrers since it is not bipartite.  

\subsection{Spanning Tree Enumeration}\label{sec:enum} There are numerous contexts (e.g., see \cite{atajan_inaba_04} or \cite{ma_18}) in which it is valuable to know the number of spanning trees $\tau(G)$ in a graph $G$.  Calculating $\tau(G)$, however, is not as trivial as one might expect.  For instance, there is no obvious systematic way to determine that the graph in Figure~\ref{fig:graphex}(a) has $11$ spanning trees without listing them out, as shown in Figure~\ref{fig:spanningex}.  For certain classes of graphs, however, it is possible to determine the number of spanning trees from the number of vertices in a graph or the degrees of its vertices.    

We proceed to survey several well known spanning tree enumeration formulae from throughout the literature beginning with Cayley's Formula for the number of spanning trees in a complete graph. While commonly attributed to Cayley~\cite{cayley}, this formula was first proved by Borchardt~\cite{Borchardt} and has since seen numerous different proofs throughout the enumerative combinatorics literature.  See the books by Aigner and Ziegler~\cite{Aigner-Ziegler} and Stanley~\cite{EC1} for more details.

\begin{theorem}[Cayley's Formula]\label{thm:completeenum}
The number of spanning trees in the complete graph on $n$ vertices is $$\tau (K_n) = n^{n-2}.$$
\end{theorem}

\begin{example}
By Theorem~\ref{thm:completeenum}, the complete graph on $4$ vertices has $\tau(K_4) = 4^{4-2}=16$ spanning trees, which we illustrate in Figure~\ref{fig:completeex} of the appendix.
\end{example}

Cayley's Formula has been generalized to complete multipartite graphs with different proofs by Austin~\cite{Austin}, Klee and Stamps~\cite{klee_stamps}, Lewis~\cite{Lewis}, and Onodera~\cite{Onodera}.

\begin{theorem}\label{thm:cmenum}
Let $n_1,\dots,n_k$ be positive integers and let $n=n_1 + \cdots + n_k.$ The number of spanning trees in the complete multipartite graph $K_{n_1,\dots,n_k}$ is
$$\tau(K_{n_1,\dots,n_k}) = n^{k-2} \prod_{i=1}^{k}(n-n_i)^{n_i - 1}.$$
\end{theorem}

In the special case of complete bipartite graphs, the formula in Theorem~\ref{thm:cmenum} can be simplified with distinct proofs by Hartsfield and Werth~\cite{Hartsfield-Werth} and Scoins~\cite{Scoins}.

\begin{theorem}\label{thm:bipartiteenum}
The number of spanning trees in the complete bipartite graph $K_{m,n}$ is $$\tau(K_{m,n})=m^{n-1}n^{m-1}.$$
\end{theorem}

\begin{example}
By Theorem~\ref{thm:bipartiteenum}, the complete bipartite graph $K_{2,3}$ has $\tau(K_{2,3}) = 2^{3-1}\cdot3^{2-1} = 12$ spanning trees, which we illustrate in Figure~\ref{fig:bipartiteex} of the appendix.
\end{example}

The following formula for the number of spanning trees in a threshold graph in terms of its vertex degrees has various proofs by Chestnut and Fishkind~\cite{Chestnut-Fishkind}, Hammer and Kelmans~\cite{Hammer-Kelmans}, Klee and Stamps~\cite{klee_stamps}, and Merris~\cite{Merris}.

\begin{theorem}\label{thm:thresholdenum}
  Let $G = (V,E)$ be a threshold graph with $V = \{v_1,\ldots,v_n\}$. Then
  $$\tau(G) = \frac{\prod_{v_i \in D}\left(\deg(v_i)+1\right) \prod_{v_j \in I}\deg(v_j)}{n}$$ where $D,I \subseteq V$ are the subsets of dominating and isolated vertices, respectively, in $G$.
\end{theorem}

\begin{example} By Theorem~\ref{thm:thresholdenum}, the threshold graph $G$ in Figure~\ref{fig:graphex}(d) has 
$$\tau(G) = \frac{\deg(v_2) \cdot \left(\deg(v_3)+1\right) \cdot \deg(v_4) \cdot \left(\deg(v_5) + 1\right)}{5} = \frac{2 \cdot (3+1) \cdot 1 \cdot (4+1)}{5} = 8$$ spanning trees, which we illustrate in Figure~\ref{fig:thresholdex} of the appendix.
\end{example}

Ferrers graphs have a similarly beautiful formula, first shown by Ehrenborg and van Willigenburg~\cite{E-VW} with a different proof by Klee and Stamps~\cite{klee_stamps}.

\begin{theorem}\label{thm:ferrersenum}
  Let $G = (V,E)$ be a Ferrers graph whose vertex set is partitioned as $V = R\sqcup C$, with $R = \{r_1,\dots,r_m\}$ and $C = \{c_1,\dots,c_n\}.$ Then
  $$\tau(G) = \frac{\prod_{i = 1}^m \deg(r_i) \prod_{j = 1}^n \deg(c_j)}{m\cdot n}.
  $$
\end{theorem}

\begin{example}
By Theorem~\ref{thm:ferrersenum}, the Ferrers graph $G$ in Figure~\ref{fig:graphex}(e) has
$$\tau(G) = \frac{\prod_{i=1}^4 \deg(r_1) \cdot \prod_{j=1}^3 \deg(c_j)}{4\cdot 3} = \frac{3\cdot2\cdot2\cdot1\cdot4\cdot3\cdot1}{12} = 12$$
spanning trees, which we illustrate in Figure~\ref{fig:ferrersex} of the appendix.
\end{example}

\subsection{Linear Algebraic Techniques}  It can be very convenient to encode the information of a graph in a matrix.  The \textbf{Laplacian matrix} of a graph $G = (V,E)$ with $V = \{v_1,\ldots,v_n\}$ is the $n \times n$ matrix $L = L(G)$ with entries given by
\[L(i,j) = \begin{cases}
    \deg(v_i) \hspace{0.5cm} &\text{if }i = j,\\
    -1 \hspace{0.5cm} &\text{if } \{v_i,v_j\} \in E, \\
    0 \hspace{0.5cm} &\text{otherwise}.
\end{cases}\]

The powerful and elegant Matrix-Tree Theorem of Kirchhoff~\cite{kirchhoff} asserts that the number of spanning trees in a graph is given by any cofactor of its Laplacian matrix.

\begin{theorem}[Matrix-Tree Theorem]\label{kirchhoff}
Let $G$ be a graph on $n$ vertices with Laplacian matrix~$L$. Then 
$$\tau(G)=(-1)^{i+j} \det(L_{i,j})$$ for any $1 \leq i,j \leq n$ where $L_{i,j}$ denotes the $(n-1) \times (n-1)$ matrix obtained by deleting the $i$th row and $j$th column from $L$.
\end{theorem}

\begin{example}
The Laplacian matrix of the graph $G$ in Figure~\ref{fig:graphex}(a) is
$$\begin{bmatrix}
2 & -1 & 0 & -1 & 0 & 0 \\
-1 & 4 & -1 & 0 & -1 & -1 \\
0 & -1 & 1 & 0 & 0 & 0 \\
-1 & 0 & 0 & 2 & -1 & 0 \\
0 & -1 & 0 & -1 & 3 & -1 \\
0 & -1 & 0 & 0 & -1 & 2
\end{bmatrix}$$
and the Matrix-Tree Theorem, applied with $i = j = 1$, asserts that
$$ \tau(G) = (-1)^{1+1}\begin{vmatrix}
4 & -1 & 0 & -1 & -1 \\
-1 & 1 & 0 & 0 & 0 \\
0 & 0 & 2 & -1 & 0 \\
-1 & 0 & -1 & 3 & -1 \\
-1 & 0 & 0 & -1 & 2
\end{vmatrix}= 11,$$
as illustrated in Figure~\ref{fig:spanningex}.
\end{example}

While the Matrix-Tree Theorem provides a systematic approach for counting the number spanning trees in a graph $G$, the determinant calculation can be algebraically and/or computationally cumbersome.  The following result due to Klee and Stamps~\cite{klee_stamps} takes an alternate approach to computing $\tau(G)$ by considering the determinants of rank-one perturbations rather than submatrices of $L(G)$. 

\begin{lemma}\label{lem:klee-stamps}
  Let $G$ be a graph on $n$ vertices with Laplacian matrix $L$, and let $\va=\langle a_i \rangle_{i \in [n]}$ and $\vb=\langle b_i \rangle_{i \in [n]}$ be column vectors in $\mathbb{R}^n.$ Then
  $$\det\left(L + \va\vb^T\right)=\left(\sum_{i=1}^n a_i\right) \cdot \left(\sum_{i = 1}^n b_i\right) \cdot \tau (G).$$
\end{lemma}

In the special case that $\va$ and $\vb$ are both equal to the all-ones vector in $\mathbb{R}^n$ (making $\va\vb^T$ the $n \times n$ all-ones matrix), this result is credited to Temperley~\cite{Temperley}. The advantage of Lemma~\ref{lem:klee-stamps} is that a clever choice of vectors $\va$ and $\vb$ can yield a rank-one perturbation $L+\va\vb^T$ whose determinant is very easy to calculate.  

\begin{example}\label{ex:threshold}
Consider the threshold graph $G$ in Figure~\ref{fig:threshold}.  If $\va = \langle 0,0,1,0,1 \rangle$ and $\vb = \langle 1,1,1,1,1 \rangle$, then the rank-one perturbation $L(G) + \va\vb^T$ calculated below is upper triangular with diagonal entries $2$, $2$, $4$, $1$, and $5$:
$$
\begin{bmatrix}
2 & 0 & -1 & 0 & -1\\
0 & 2 & -1 & 0 & -1\\
-1 & -1 & 3 & 0 & -1\\
0 & 0 & 0 & 1 & -1\\
-1 & -1 & -1 & -1 & 4\\
\end{bmatrix} + \begin{bmatrix}
0 & 0 & 0 & 0 & 0\\
0 & 0 & 0 & 0 & 0\\
1 & 1 & 1 & 1 & 1\\
0 & 0 & 0 & 0 & 0\\
1 & 1 & 1 & 1 & 1\\
\end{bmatrix}
= \begin{bmatrix}
2 & 0 & -1 & 0 & -1\\
0 & 2 & -1 & 0 & -1\\
0 & 0 & 4 & 1 & 0\\
0 & 0 & 0 & 1 & -1\\
0 & 0 & 0 & 0 & 5\\
\end{bmatrix}
$$
Since $\det(L(G)+\va\vb^T) = 2 \cdot 2 \cdot 4 \cdot 1 \cdot 5 = 80$, Lemma~\ref{lem:klee-stamps} asserts that $\tau(G) = 80 / (2\cdot 5) = 8$, as illustrated in Figure~\ref{fig:thresholdex}.
\end{example}

\begin{example}\label{ex:Ferrers}
Consider the Ferrers graph $G$ in Figure~\ref{fig:ferrers} with its vertices ordered according to its Ferrers traversal: $v_1 = c_1$, $v_2 = r_4$, $v_3 = c_2$, $v_4 = r_3$, $v_5 = r_2$, $v_6 = c_3$, and $v_7 = r_1$.  If $\va = \langle 0,1,0,1,1,0,1 \rangle$ and $\vb = \langle 1,0,1,0,0,1,0 \rangle$, then the rank-one perturbation $L(G) + \va\vb^T$ calculated below is upper triangular with diagonal entries $4$, $1$, $3$, $2$, $2$, $1$, $3$:
$$\begin{small}
\begin{bmatrix}
4 & -1 & 0 & -1 & -1 & 0 & -1\\
-1 & 1 & 0 & 0 & 0 & 0 & 0\\
0 & 0 & 3 & -1 & -1 & 0 & -1\\
-1 & 0 & -1 & 2 & 0 & 0 & 0\\
-1 & 0 & -1 & 0 & 2 & 0 & 0\\
0 & 0 & 0 & 0 & 0 & 1 & -1\\
-1 & 0 & -1 & 0 & 0 & -1 & 3
\end{bmatrix} + \begin{bmatrix}
0 & 0 & 0 & 0 & 0 & 0 & 0\\
1 & 0 & 1 & 0 & 0 & 1 & 0\\
0 & 0 & 0 & 0 & 0 & 0 & 0\\
1 & 0 & 1 & 0 & 0 & 1 & 0\\
1 & 0 & 1 & 0 & 0 & 1 & 0\\
0 & 0 & 0 & 0 & 0 & 0 & 0\\
1 & 0 & 1 & 0 & 0 & 1 & 0\\
\end{bmatrix}
= \begin{bmatrix}
4 & -1 & 0 & -1 & -1 & 0 & -1\\
0 & 1 & 1 & 0 & 0 & 1 & 0\\
0 & 0 & 3 & -1 & -1 & 0 & -1\\
0 & 0 & 0 & 2 & 0 & 1 & 0\\
0 & 0 & 0 & 0 & 2 & 1 & 0\\
0 & 0 & 0 & 0 & 0 & 1 & -1\\
0 & 0 & 0 & 0 & 0 & 0 & 3
\end{bmatrix} \end{small}
$$
Since $\det(L(G)+\va\vb^T) = 4 \cdot 1 \cdot 3 \cdot 2 \cdot 2 \cdot 1 \cdot 3 = 144$, Lemma~\ref{lem:klee-stamps} asserts that $\tau(G) = 144 / (4\cdot 3) = 12$, as illustrated in Figure~\ref{fig:ferrersex}.
\end{example}

The ideas in Examples~\ref{ex:threshold} and \ref{ex:Ferrers} can be extended to give new straightforward proofs of Theorems~\ref{thm:thresholdenum} and \ref{thm:ferrersenum}.

\begin{proof}[Proof of Theorem~\ref{thm:thresholdenum}]
  Let $v_1,\dots,v_n$ be the ordering of vertices in $V$ given by the construction order of $G$.  We claim that the rank-one perturbation $L(G)+\va\vb^T$, where $\va \in \mathbb{R}^n$ is the indicator vector of the dominating vertices in $G$, that is, 
  \[
    a_{i}=\begin{cases}
    1 & \text{if } v_i \in D,\\
    0 & \text{otherwise},
    \end{cases}
  \]
  and $\vb \in \mathbb{R}^n$ is the all-ones vector, is upper triangular. To see why, observe that for $1 \leq j < i \leq n$, $(\va\vb^T)(i,j)=1$ if and only if $v_i$ is a dominating vertex if and only if $\{v_i,v_j\} \in E$ if and only if $L(i,j)=-1$.
  Since $L(G) + \va\vb^T$ is upper triangular, its determinant is simply the product of its diagonal entries, which are
    \[
    \left(L+\va\vb^{T}\right)\left(i,i\right)=
    \begin{cases}
    \deg\left(v_{i}\right)+1 & \text{if } v_i \in D,\\
    \deg\left(v_{i}\right) & \text{otherwise.}
    \end{cases}
  \]
  Thus,
  \[
  \det\left(L+\va\vb^{T}\right) = \deg(v_1)\cdot \prod_{i \in D}\left(\deg(v_i)+1\right) \cdot \prod_{i \in I}\deg(v_i).
  \]
  Since $N(v_1) = D$, $\deg(v_1) = |D| = \sum\limits_{i = 1}^n a_i$, and the desired expression for $\tau(G)$ follows directly from Lemma~\ref{lem:klee-stamps}.
  \end{proof}

\begin{proof}[Proof of Theorem~\ref{thm:ferrersenum}]
Let $v_1,\dots,v_{m+n}$ be the ordering of the vertices in $V$ given by the Ferrers traversal of $G$. We claim that the rank-one perturbation $L(G)+\va\vb^T$, where $\va, \vb \in \mathbb{R}^{m+n}$ are the indicator vectors of $R$ and $C$, respectively, that is,
$$a_i = \begin{cases}1 & \text{if }v_i\in R,\\ 0 &\text{else},\end{cases}\qquad b_i = \begin{cases}1 & \text{if }v_i\in C,\\ 0 &\text{else},\end{cases}$$ is upper triangular.  To see why, observe that for $1 \leq j < i \leq m+n$, $(\va\vb^T)(i,j)=1$ if and only if $v_i \in R$ and $v_j \in C$ if and only if $\{v_i,v_j\} \in E$ if and only if $L(i,j)=-1$.
  Since $L(G) + \va\vb^T$ is upper triangular, its determinant is simply the product of its diagonal entries, which are precisely the diagonal entries of $L(G)$.  (The diagonal entries of $\va\vb^T$ are all zero since no vertex belongs to both $R$ and $C$.) Thus,
$$  \det\left(L(G) + \va\vb^T\right) = \prod_{k = 1}^{m+n}\deg(v_k) = \prod_{i=1}^m \deg(r_i) \cdot \prod_{j = 1}^n \deg(c_j).$$
Since the entries of $\va$ and $\vb$ sum to $|R| = m$ and $|C| = n$, respectively, the desired expression for $\tau(G)$ follows directly from Lemma~\ref{lem:klee-stamps}.
\end{proof}

These new proofs of Theorems~\ref{thm:thresholdenum} and \ref{thm:ferrersenum} share a common key ingredient -- the vertices of any threshold or Ferrers graph can be ordered so its Laplacian matrix admits a triangular rank-one perturbation.  In the next section, we show that this approach can be extended to special $2$-threshold graphs.

\section{Main Results}\label{sec:nuclear}

This is the main section of our article where we prove that a graph Laplacian admits a triangular rank-one perturbation if and only if the graph is special $2$-threshold.  

\begin{theorem}\label{thm:quasiupdate}
  The Laplacian matrix of a graph $G$ admits a triangular rank-one perturbation if and only if $G$ is special $2$-threshold.
\end{theorem}

Our proof of Theorem~\ref{thm:quasiupdate} relies on the following characterization of special $2$-threshold graphs that extends Theorem~\ref{thm:threshold-order}.  

\begin{theorem}\label{thm:order}
A graph $G = (V,E)$ with $|V| = n$ is special $2$-threshold if and only if there exists an ordering $v_1,\dots,v_n$ of $V$ and a subset $U \subseteq V$ such that
\[    N^<(v_i) = \emptyset \quad \text{ or } \quad 
    N^<(v_i) = U \cap \{v_1, \ldots, v_{i-1}\}
\] for every $i \in [n]$.
\end{theorem}

As with the special case of threshold graphs, we call such an ordering of the vertices of a special $2$-threshold graph a \textbf{construction order}.  Vertices with empty lower neighborhoods in a construction order will continue to be called \textbf{isolated} while vertices with non-empty lower neighborhoods will be called \textbf{\textit{U}-dominating}.

\begin{example}\label{ex:nuclear} The special $2$-threshold graph in Figure~\ref{fig:graphex}(e) has a construction order $v_1 = 1$, $v_2 = 5$, $v_3 = 2$, $v_4 = 3$, and $v_5 = 4$ with  $U = \{v_1,v_2,v_4\}$.  Indeed, $v_1$ is an initial vertex, $v_4$ is isolated, and $v_2$, $v_3$, and $v_5$ are $U$-dominating, as illustrated below with the vertices in $U$ circled.

\begin{figure}[H]
    \centering
    \begin{tikzpicture}
        \filldraw[opacity=0.25] (0,0) circle (2pt);
        \draw[opacity=0.25] (0,0) circle (4pt);
        \node [opacity=0.25,below left] at (0,0) {$v_4$};
        \filldraw[opacity=0.25]  (1,0) circle (2pt);
        \node [opacity=0.25,below] at (1,0) {$v_5$};
        \filldraw  (1,1) circle (2pt);
        \draw (1,1) circle (4pt);
        \node [above left] at (1,1) {$v_1$};
        \filldraw[opacity=0.25]  (2,1) circle (2pt);
        \node [opacity=0.25,above right] at (2,1) {$v_3$};
        \filldraw  (2,0) circle (2pt);
        \draw (2,0) circle (4pt);
        \node [below right] at (2,0) {$v_2$};
        \draw[opacity=0.25] (0,0)--(1,0);
        \draw[opacity=0.25] (1,0)--(2,0)--(2,1)--(1,1)--(1,0);
        \draw[opacity=1] (1,1)--(2,0);
    \end{tikzpicture}
    \quad
    \begin{tikzpicture}
        \filldraw[opacity=0.25] (0,0) circle (2pt);
        \draw[opacity=0.25] (0,0) circle (4pt);
        \node [opacity=0.25,below left] at (0,0) {$v_4$};
        \filldraw[opacity=0.25] (1,0) circle (2pt);
        \node [opacity=0.25,below] at (1,0) {$v_5$};
        \filldraw (1,1) circle (2pt);
        \draw (1,1) circle (4pt);
        \node [above left] at (1,1) {$v_1$};
        \filldraw (2,1) circle (2pt);
        \node [above right] at (2,1) {$v_3$};
        \filldraw (2,0) circle (2pt);
        \draw (2,0) circle (4pt);
        \node [below right] at (2,0) {$v_2$};
        \draw[opacity=0.25] (0,0)--(1,0);
        \draw[opacity=0.25] (1,0)--(2,0)--(2,1)--(1,1)--(1,0);
        \draw[opacity=1] (1,1)--(2,0);
        \draw (1,1)--(2,1)--(2,0);
    \end{tikzpicture}
    \quad
    \begin{tikzpicture}
        \filldraw (0,0) circle (2pt);
        \draw (0,0) circle (4pt);
        \node [below left] at (0,0) {$v_4$};
        \filldraw[opacity=0.25] (1,0) circle (2pt);
        \node [opacity=0.25,below] at (1,0) {$v_5$};
        \filldraw (1,1) circle (2pt);
        \draw (1,1) circle (4pt);
        \node [above left] at (1,1) {$v_1$};
        \filldraw (2,1) circle (2pt);
        \node [above right] at (2,1) {$v_3$};
        \filldraw (2,0) circle (2pt);
        \draw (2,0) circle (4pt);
        \node [below right] at (2,0) {$v_2$};
        \draw[opacity=0.25] (0,0)--(1,0);
        \draw[opacity=0.25] (1,0)--(2,0)--(2,1)--(1,1)--(1,0);
        \draw[opacity=1] (1,1)--(2,0);
        \draw (1,1)--(2,1)--(2,0);
    \end{tikzpicture}
    \quad
    \begin{tikzpicture}
        \filldraw (0,0) circle (2pt);
        \draw (0,0) circle (4pt);
        \node [below left] at (0,0) {$v_4$};
        \filldraw (1,0) circle (2pt);
        \node [below] at (1,0) {$v_5$};
        \filldraw (1,1) circle (2pt);
        \draw (1,1) circle (4pt);
        \node [above left] at (1,1) {$v_1$};
        \filldraw (2,1) circle (2pt);
        \node [above right] at (2,1) {$v_3$};
        \filldraw (2,0) circle (2pt);
        \draw (2,0) circle (4pt);
        \node [below right] at (2,0) {$v_2$};
        \draw[opacity=1] (0,0)--(1,0);
        \draw[opacity=15] (1,0)--(2,0)--(2,1)--(1,1)--(1,0);
        \draw[opacity=1] (1,1)--(2,0);
        \draw (1,1)--(2,1)--(2,0);
    \end{tikzpicture}
    \caption{A construction order for the vertices of a special $2$-threshold graph.}
    \label{fig:quasiex}
\end{figure}
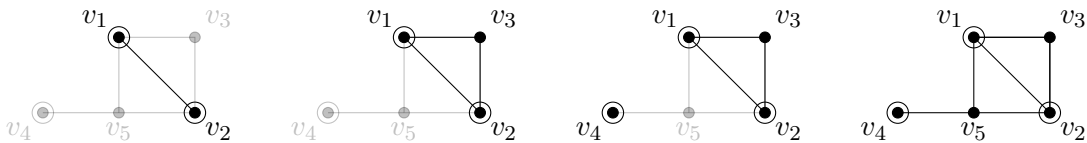
\end{example}

Together, Theorems~\ref{thm:quasiupdate} and \ref{thm:order} enable us to prove the following spanning tree enumeration formula for special $2$-threshold graphs that generalizes the well known formulae for threshold and Ferrers graphs in Theorems~\ref{thm:thresholdenum} and \ref{thm:ferrersenum}.

 \begin{theorem}\label{thm:quasienum}
    Let $G = (V,E)$ be a $U$-threshold graph with $V = \{v_1,\dots,v_n\}$. Then
    \[ \tau(G) = \frac{\prod_{v_i\in D_U} \left(\deg(v_i)+1\right) \cdot \prod_{v_j\in V \setminus D_U}\left(\deg(v_j)\right)}{|D|\cdot|U|}
    \]
    where $D \subseteq V$ is the subset of $U$-dominating vertices in $G$ and $D_U = D\cap U$.
  \end{theorem}
  
  \begin{example} Consider once more the special $2$-threshold graph $G$ in Figure~\ref{fig:graphex}(e).  With the construction order from Example~\ref{ex:nuclear}, $U = \{v_1, v_2, v_4\}$, $D = \{v_2,v_3,v_5\}$, and $D_U = \{v_2\}$. By Theorem~\ref{thm:quasienum}, $G$ has
$$\tau(G) = \frac{\deg(v_1) \cdot \left(\deg(v_2)+1\right) \cdot \deg(v_3) \cdot \deg(v_4) \cdot \deg(v_5)}{3 \cdot 3} = \frac{3 \cdot (3+1) \cdot 2 \cdot 1 \cdot 3}{9} = 8$$ spanning trees, which we illustrate in Figure~\ref{fig:quasienum} of the appendix.
\end{example}
\newpage
To see how Theorem~\ref{thm:quasienum} provides a common generalization of Theorem~\ref{thm:thresholdenum} and Theorem~\ref{thm:ferrersenum}, observe the following: 
\begin{itemize}
    \item If $G = (V,E)$ is a threshold graph on $n$ vertices, then it is $U$-threshold with $U = V$ by Remark~\ref{remark:threshold-special}, and $D = N(v_1)$.  Substituting $D_U = D$, $|D| = \deg(v_1)$, and $|U| = n$ into the formula in Theorem~\ref{thm:quasienum} -- and simplifying appropriately -- yields the formula in Theorem~\ref{thm:thresholdenum}.  
    \item If $G = (V,E)$ is a Ferrers graph with $V = R \sqcup C$, then it is $U$-threshold with $U = C$ by Proposition~\ref{prop:ferrersqt}, $D = R$ since every vertex in $R$ is adjacent to $v_1 = c_1$, and $D_U = D \cap U = R \cap C = \emptyset$.  Substituting $D = R$ and $U = C$ into the formula in Theorem~\ref{thm:quasienum} -- and simplifying appropriately -- yields the formula in Theorem~\ref{thm:ferrersenum}.
\end{itemize}

We proceed with the proofs of Theorems~\ref{thm:quasiupdate} and \ref{thm:quasienum} and leave the proof of Theorem~\ref{thm:order} for the end of the section since it has numerous components, none of which are directly referenced in the proofs of Theorems~\ref{thm:quasiupdate} and \ref{thm:quasienum}.

\begin{proof}[Proof of Theorem~\ref{thm:quasiupdate}.]
  Starting with the reverse direction, let $G = (V,E)$ be a special $2$-threshold graph on $n$ vertices.  By Theorem~\ref{thm:order}, there exists a construction order $v_1,\ldots,v_n$ of $V$ for some $U \subseteq V$.  Then the rank-one perturbation $L(G)+\va\vb^T$ where $\va$ and $\vb$ are the indicator vectors of the subset $D$ of $U$-dominating vertices and $U$, respectively, that is,
$$a_i = \begin{cases}1 & \text{if }v_i\in D,\\ 0 &\text{else},\end{cases}\qquad b_i = \begin{cases}1 & \text{if }v_i\in U,\\ 0 &\text{else},\end{cases}$$ is upper triangular since for $1 \leq j < i \leq n$, $(\va\vb^T)(i,j)=1$ if and only if $v_i \in D$ and $v_j \in U$ if and only if $\{v_i,v_j\} \in E$ if and only if $L(G)(i,j)=-1$.
   
  For the forward direction, suppose $G = (V,E)$ is a graph with $V = \{v_1,\dots,v_n\}$ such that $L(G)$ admits an upper triangular rank-one perturbation, that is, there exists a rank-one matrix $M \in \mathbb{R}^{n \times n}$ such that $L(G) + M$ is upper-triangular. Since $M$ is rank-one, it can be factored into an outer product $M = \va\vb^T$ of two column vectors $\va = \langle a_i \rangle$ and $\vb= \langle b_i \rangle$ in $\mathbb{R}^n$. Moreover, since $L(G) + M$ is upper triangular and $L(G)(i,j) \in \{-1,0\}$ for $1 \leq j < i \leq n$, it follows that $M(i,j) \in \{0,1\}$ for $1 \leq j < i \leq n$, and we may assume without loss of generality that $M \in \{0,1\}^{n \times n}$ and $\va , \vb \in \{0,1\}^n$.  Let $D = \{v_i \in V \ | \ a_i = 1\}$ and $U = \{v_i \in V \ | \ b_i = 1\}$.  If $v_i \in V \setminus D$, then $a_i = 0$, which means $L(G)(i,j) = M(i,j) = 0$ for $1 \leq j < i \neq n$, hence $N^<(v_i) = \emptyset$.  If $v_i \in D$, on the other hand, then $a_i = 1$, which means $$M(i,j) = \begin{cases} 1 & \text{if } v_j \in U, \\ 0 & \text{else},\end{cases} \quad \Longleftrightarrow \quad L(G)(i,j) = \begin{cases} -1 & \text{if } v_j \in U, \\ 0 & \text{else},\end{cases}$$ hence $N^<(v_i) = U \cap \{v_1, \ldots, v_{i-1}\}.$  Therefore, $v_1, \ldots, v_n$ is a construction order for $U$, and $G$ is special $2$-threshold by Theorem~\ref{thm:order}.
  \end{proof}

\begin{proof}[Proof of Theorem~\ref{thm:quasienum}.]
  From the proof of Theorem~\ref{thm:quasiupdate}, the rank-one perturbation $L(G)+\va\vb^T$ where
$$a_i = \begin{cases}1 & \text{if }v_i\in D,\\ 0 &\text{else},\end{cases}\qquad b_i = \begin{cases}1 & \text{if }v_i\in U,\\ 0 &\text{else},\end{cases}$$ is upper triangular. Since the diagonal entries of $L(G)+\va\vb^{T}$ are given by
  \[\left(L(G)+\va\vb^{T}\right)\left(i,i\right)=
    \begin{cases}
    \deg\left(v_{i}\right)+1 & \text{if } v_i \in D_U,\\
    \deg\left(v_{i}\right) & \text{else.}
    \end{cases}
  \]
  and $\va$ and $\vb$ are the indicator vectors for $D$ and $U$ in $V$, respectively, we know that 
  \[
  \det\left(L(G)+\va\vb^{T}\right) = \prod_{v\in D_U} \left(\deg(v)+1\right) \prod_{v\in V \setminus D_U}\left(\deg(v)\right),
  \]
  $\sum_{i=1}^n a_i = |D|$, and $\sum_{i=1}^n b_i = |U|$.  The desired expression for $\tau(G)$ follows directly from Lemma~\ref{lem:klee-stamps}.
 \end{proof} 

It remains to prove Theorem~\ref{thm:order}.  Since every special $2$-threshold graph $G = (V,E)$ is $U$-threshold for some $U \subseteq V$, it suffices to show that $G$ can be constructed from an initial vertex by repeatedly appending isolated or $U$-dominating vertices.  Because the induced subgraph $G[U]$ is threshold, our approach to proving Theorem~\ref{thm:order} will be to extend a threshold construction order on $U$ to a special $2$-threshold order on $V$.  A challenge with this approach, however, is that construction orders of threshold graphs are not unique.  For instance, any~ordering of the vertices of a complete graph yields a valid construction order.  Fortunately, the proof of Theorem~\ref{thm:threshold-order} in \cite{chvatal_hammer_1977} is constructive, and it is straightforward to check that vertices with the same degree appear consecutively in any construction order and that construction orders are unique up to permutations of vertices that share the same degree.  In other words, the construction orders of a threshold graph $G = (V,E)$ induce a common total order $\preceq$ on $V/\sim$ where $v \sim w$ if and only if $\deg(v) = \deg(w)$ for every $v,w \in V$. See Example~\ref{ex:special-order} for an illustration. 

Building upon this framework, we prove Theorem~\ref{thm:order} by refining and extending the order $\preceq$ from threshold to special $2$-threshold graphs.  Let $G = (V,E)$ be a $U$-threshold graph for some $U \subseteq V$, let $U^c = V \setminus U$, and consider the equivalence relation on $V$ given by $x \sim y$ if and only if $\deg_U(x) = \deg_U(y)$, $\deg_{U^c}(x) = \deg_{U^c}(y)$, and either $x,y \in U$ or $x,y \in U^c$.  Since the induced subgraph $G[U]$ is threshold ( Remark~\ref{remark:threshold-special}), there is a unique total order $\preceq$ on $U/\sim$, which we refine and extend to an order relation $\trianglelefteq$ on $V/\sim$ given by $x \trianglelefteq y$ if and only if one of the following conditions hold:
\smallskip
\begin{itemize} \itemsep 4pt \item $x,y \in U$, $x \preceq y$, and $\deg_{U^c}(x) \geq \deg_{U^c}(y)$, \item $x 
\in U$, $y \in U^c$, and $\{x,y\} \in E$, \item $x \in U^c$, $y \in U$, and $\{x,y\} \notin E$, \item $x,y \in U^c$ and $\deg_U(x) \leq \deg_U(y)$. \end{itemize} 
\smallskip
We claim that $\trianglelefteq$ is a total order on $V/\sim$, but before we prove it, let us first illustrate $\preceq$ and $\trianglelefteq$ with an example and establish a lemma regarding neighborhoods in $U$-threshold graphs.

\begin{example}\label{ex:special-order}
Consider the graph $G = (V,E)$ in Figure~\ref{fig:special}.  It is special $2$-threshold by Theorem~\ref{conj:forbidden}.  We leave it to the reader to verify that $G$ is $U$-threshold for $U = \{a,b,c,d,e,f\} \subseteq V$, which implies that the induced subgraph $G[U]$ is threshold.  (The thresholdness of $G[U]$ can also be verified with Theorem~\ref{thm:thresholdforbidden}).  Since $\deg_U(a) = \deg_U(c) = \deg_U(f) = 1$, $\deg_U(b) = 5$, and $\deg_U(d) = \deg_U(e) = 2$, the elements of $U/\sim$ are $\{a,c,f\}$, $\{b\}$, and $\{d,e\}$, and the algorithm in \cite{chvatal_hammer_1977} yields the total order $\{d,e\} \preceq \{a,c,f\} \preceq \{b\}$ on $U/\sim$.  It follows that, in every construction order for $G[U]$, vertices $d$ and $e$ appear in the first two positions (in either permutation), vertices $a$, $c$, and $f$ appear in the third through fifth positions (in any permutation), and vertex $b$ appears in the sixth position.  

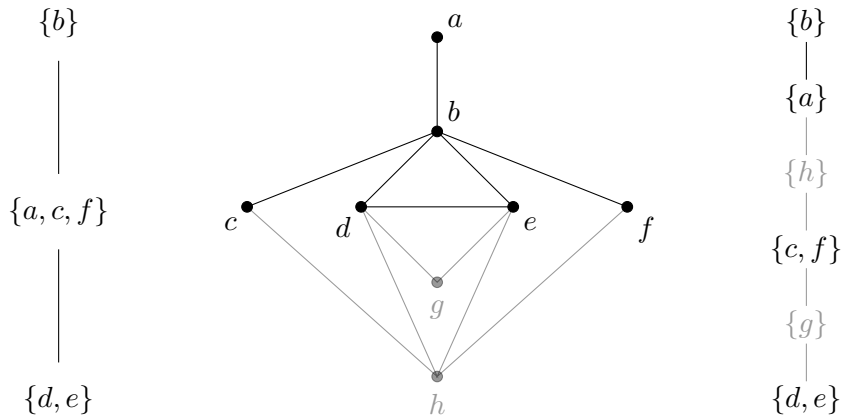
\begin{figure}[H]
    \centering
    \begin{tikzpicture}
        \node [opacity=1] at (0,-2.5) {$\{d,e\}$};
        \node [opacity=1] at (0,0) {$\{a,c,f\}$};
        \node [opacity=1] at (0,2.5) {$\{b\}$};
        \draw[opacity=1] (0,-2)--(0,-0.5);
        \draw[opacity=1] (0,0.5)--(0,2);
    \end{tikzpicture}
    \hspace{1cm}
    \begin{tikzpicture}
        \filldraw[opacity=1] (-2.5,0) circle (2pt);
        \node [opacity=1,below left] at (-2.5,0) {$c$};
        \filldraw[opacity=1]  (-1,0) circle (2pt);
        \node [opacity=1,below left] at (-1,0) {$d$};
        \filldraw  (1,0) circle (2pt);
        \node [below right] at (1,0) {$e$};
        \filldraw[opacity=1]  (2.5,0) circle (2pt);
        \node [opacity=1,below right] at (2.5,0) {$f$};
        \filldraw[opacity=0.4] (0,-2.25) circle (2pt);
        \node [opacity=0.4,below] at (0,-2.35) {$h$};
        \filldraw[opacity=0.4]  (0,-1) circle (2pt);
        \node [opacity=0.4,below] at (0,-1.1) {$g$};
        \filldraw  (0,1) circle (2pt);
        \node [above right] at (0,1) {$b$};
        \filldraw[opacity=1]  (0,2.25) circle (2pt);
        \node [opacity=1,above right] at (0,2.25) {$a$};
        \draw[opacity=1] (0,2.25)--(0,1);
        \draw[opacity=1] (-2.5,0)--(0,1)--(-1,0)--(1,0)--(0,1)--(2.5,0);
        \draw[opacity=0.4] (-2.5,0)--(0,-2.25)--(-1,0)--(0,-1)--(1,0)--(0,-2.25)--(2.5,0);
    \end{tikzpicture}
    \hspace{1cm}
    \begin{tikzpicture}
        \node [opacity=1] at (0,-2.5) {$\{d,e\}$};
        \node [opacity=0.4] at (0,-1.5) {$\{g\}$};
        \node [opacity=1] at (0,-0.5) {$\{c,f\}$};
        \node [opacity=0.4] at (0,0.5) {$\{h\}$};
        \node [opacity=1] at (0,1.5) {$\{a\}$};
        \node [opacity=1] at (0,2.5) {$\{b\}$};
        \draw[opacity=0.4] (0,-2.25)--(0,-1.75);
        \draw[opacity=0.4] (0,-1.25)--(0,-0.75);
        \draw[opacity=0.4] (0,-0.25)--(0,0.25);
        \draw[opacity=0.4] (0,0.75)--(0,1.25);
        \draw[opacity=1] (0,1.75)--(0,2.25);
    \end{tikzpicture}
    \caption{A $U$-threshold graph (center) and Hasse diagrams of the total orders $\preceq$ on $U/\sim$ (left) and $\trianglelefteq$ on $V/\sim$ (right).}
    \label{fig:special}
\end{figure}

Furthermore, since $\deg_{U^c}(a) = \deg_{U^c}(b) = \deg_{U^c}(g) = \deg_{U^c}(h)= 0$, $\deg_{U^c}(c) = \deg_{U^c}(f) = 1$, $\deg_{U^c}(d) = \deg_{U^c}(e) = 2$, $\deg_U(g) = 2$, and $\deg_U(h) = 4$, the elements of $V/\sim$ are $\{a\}$, $\{b\}$, $\{c,f\}$, $\{d,e\}$, $\{g\}$, and $\{h\}$.  Note that the element $\{a,c,f\}$ in $U/\sim$ is partitioned into $\{a\}$ and $\{c,f\}$ in $V/\sim$ since $\deg_{U^c}(a) \neq \deg_{U^c}(c) = \deg_{U^c}(f)$.  Applying the conditions in the definition of $\trianglelefteq$ to the elements of $V/\sim$ yields the ordering $\{d,e\} \trianglelefteq \{g\} \trianglelefteq \{c,f\} \trianglelefteq \{h\} \trianglelefteq \{a\} \trianglelefteq \{b\}$.  For instance, $\{h\}$ succeeds $\{d,e\}$, $\{g\}$, and $\{c,f\}$ and precedes $\{a\}$ and $\{b\}$ since $\deg_U(g) \leq \deg_{U}(h)$ and $h$ is adjacent to $c$, $d$, $e$, and $f$ but not $a$ or $b$.  One can check that any ordering of $V$ in which vertices $d$ and $e$ appear in the first two positions (in either permutation), vertices $c$ and $f$ appear in the fourth and fifth positions (in either permutation), and vertices $g$, $h$, $a$ and $b$ appear in the third, sixth, seventh, and eigth positions, respectively, is a valid construction order for $G$.
\end{example}
 
 \begin{lemma}\label{lem:nesting}
 If $G = (V,E)$ is a $U$-threshold graph for some $U \subseteq V$ and $U^c = V \setminus U$, then \begin{enumerate}[\quad (a)]
 \item $U^c$ is an independent set in $G$,
 \item the neighborhoods of the vertices in $U^c$ are nested,
 \item the neighborhoods of the vertices in $U$ restricted to $U^c$ are nested, and
 \item $N_{U^c}(v) \supseteq N_{U^c}(w)$ for every $v,w \in U$ such that $v \prec w$ in $G[U]$.
 \end{enumerate}
 \end{lemma}
 
 \begin{proof}
Let $G = (V,E)$ be a $U$-threshold graph for some $U \subseteq V$.

For part (a), suppose $v,w \in U^c$ and consider the induced subgraph $G[W]$ with $W = \{v,w\}$.  Since neither $v$ nor $w$ belongs to $U$, one (in fact each) of them must be an isolated vertex in $G[W]$ by the definition of $U$-threshold, hence $\{v,w\} \notin E$.  Since there is no edge between any pair of vertices in $U^c$, $U^c$ is an independent set in $G$.

For part (b), suppose $v,w \in U^c$ such that $N(v) \neq \emptyset$ and $N(w) \neq \emptyset$, and consider the induced subgraph $G[W]$ where $W = \{v,w\} \cup N(v) \cup N(w)$.  By part (a), $N(v) \cup N(w) \subseteq U$, which means $N(v) \cup N(w) = U \cap W$.  Since every vertex in $N(v) \cup N(w)$ is adjacent to $v$ or $w$, $G[W]$ does not have any isolated vertices. Moreover, $N_{W}(x) \neq (W \setminus \{x\}) \cap U$ for every $x \in N(v) \cup N(w)$.  Therefore, since $G$ is $U$-threshold, it must be the case that $N(v) = U \cap W = N(v) \cup N(w)$ or $N(w) = U \cap W = N(v) \cup N(w)$, which means $N(v) \subseteq N(w)$ or $N(w) \subseteq N(v)$.

For part (c), let $v,w \in U$ such that $N_{U^c}(v) \neq \emptyset$ and $N_{U^c}(w) \neq \emptyset$, and consider the induced subgraph $G[W]$ with $W = \{v,w\} \cup (N_{U^c}(v) \cup N_{U^c}(w)) \setminus (N_{U^c}(v) \cap N_{U^c}(w))$. Since every vertex in $W \setminus \{v,w\}$ is adjacent to exactly one of $v$ or $w$, it must be the case that one of $v$ or $w$ is an isolated or $U$-dominating vertex in $G[W]$, which implies that one of $v$ or $w$ is not adjacent to any vertices in $W \setminus \{v,w\}$.   Therefore, $N_{U^c}(v) \subseteq N_{U^c}(w)$ or $N_{U^c}(w) \subseteq N_{U^c}(v)$. 

For part (d), let $v,w \in U$ such that $v \prec w$ in the induced threshold graph $G[U]$.  By part (c), either $N_{U^c}(v) \subseteq N_{U^c}(w)$ or $N_{U^c}(v) \supseteq N_{U^c}(w)$.  Suppose for contradiction that $N_{U^c}(v) \subsetneq N_{U^c}(w)$ and let $z \in N_{U^c}(w) \setminus N_{U^c}(v)$.  There are two cases to consider based on whether or not $v$ and $w$ are adjacent: either (i) $\{v,w\} \in E$ and there exists $u \in U$ such that $\{u,v\} \notin E$ and $\{u,w\} \in E$ or (ii) $\{v,w\} \notin E$ and there exists $u \in U$ such that $\{u,v\} \in E$ and $\{u,w\} \notin E$.  For each case, it is straightforward to check that the induced subgraph $G[W]$ on $W = \{v,w,z,u\}$ fails to be $U$-threshold, which contradicts the assumptions that $G$ is $U$-threshold and there exists $z \in N_{U^c}(w) \setminus N_{U^c}(v)$.  It must, therefore, be the case that $N_{U^c}(v) \supseteq N_{U^c}(w)$.
 \end{proof} 

\begin{proposition}
If $G = (V,E)$ is a $U$-threshold graph for some $U \subseteq V$, then the relation $\trianglelefteq$ is a total order on $V/\sim$.
\end{proposition}

\begin{proof} Let $G = (V,E)$ be a $U$-threshold graph for some $U \subseteq V$.  We will show that $\trianglelefteq$ is reflexive, antisymmetric, transitive, and total on $V/\sim$.
\smallskip
\begin{itemize} \itemsep 4pt
\item (Reflexivity) Suppose $x,y \in V$ such that $x \sim y$.  Then $\deg_U(x) = \deg_U(y)$, $\deg_{U^c}(x) = \deg_{U^c}(y)$, and either $x,y \in U$ or $x,y \in U^c$.  Since $x \sim y$ in $V$ implies $x \sim y$ in $U$ for any $x,y \in U$, it follows that $x \trianglelefteq y$ in $V$.  

\item (Antisymmetry) Suppose $x,y \in V$ such that $x \trianglelefteq y$ and $x \trianglerighteq y$.  Then $x,y \in U$ or $x,y \in U^c$; otherwise, $\{x,y\} \in E$ and $\{x,y\} \notin E$, which is impossible.  If $x,y \in U^c$, then $\deg_{U}(x) = \deg_U(y)$ by the definition of $\trianglelefteq$ and $\deg_{U^c}(x) = \deg_{U^c}(y) = 0$ by Lemma~\ref{lem:nesting}(a).  Similarly, if $x,y \in U$, then $\deg_{U}(x) = \deg_{U}(y)$ and $\deg_{U^c}(x) = \deg_{U^c}(y)$ by the definitions of $\sim$ on $U$ and $\trianglelefteq$. In either case, $x \sim y$ in $V$.  

\item (Transitivity) Suppose $x,y,z \in V$ such that $x \trianglelefteq y$ and $y \trianglelefteq z$. We consider the eight possible combinations of memberships for $x$, $y$, and $z$ in $U$ and $U^c$:

\begin{itemize} \itemsep 2pt
    \item If $x, y, z \in U$, then $x \preceq y \preceq z$ and $\deg_{U^c}(x) \geq \deg_{U^c}(y) \geq \deg_{U^c}(z)$. Thus, $x \trianglelefteq z$ by the transitivity of $\preceq$ and $\geq$. 
    \item If $x, y \in U$ and $z \in U^c$, then $\deg_{U^c}(x) \geq \deg_{U^c}(y)$ and $\{y,z\} \in E$.  By Lemma~\ref{lem:nesting}(c), $N_{U^c}(x) \supseteq N_{U^c}(y)$.  Therefore, $z \in N_{U^c}(y) \subseteq N_{U^c}(x)$, which means $\{x,z\} \in E$ and $x \trianglelefteq z$.
    \item If $x, z \in U$ and $y \in U^c$, then $\{x, y\} \in E$ and $\{y, z\} \notin E$.  Since $y \in N_{U^c}(x)$ and $y \notin N_{U^c}(z)$, and $N_{U^c}(x)$ and $N_{U^c}(z)$ are nested by Lemma~\ref{lem:nesting}(c), it must be the case that $\deg_{U^c}(x) \geq \deg_{U^c}(z)$.  Moreover, $x \preceq z$ by Lemma~\ref{lem:nesting}(d), hence $x \trianglelefteq z$.
    \item If $y, z \in U$ and $x \in U^c$, then $\{x,y\} \notin E$ and $\deg_{U^c}(y) \geq \deg_{U^c}(z)$. By Lemma~\ref{lem:nesting}(c), $N_{U^c}(y) \supseteq N_{U^c}(z)$.  Therefore, $x \notin N_{U^c}(y) \supseteq N_{U^c}(z)$, which means $\{x,z\} \notin E$ and $x \trianglelefteq z$.
    \item If $x \in U$ and $y, z \in U^c$, then $\{x, y\} \in E$ and $\deg_U(y) \leq \deg_U(z)$. By Lemma~\ref{lem:nesting}(a-b), $N(y) \subseteq N(z)$. Therefore, $x \in N(y) \subseteq N(z)$, which means $\{x,z\} \in E$ and $x \trianglelefteq z$.
    \item If $y \in U$ and $x, z \in U^c$, then $\{x,y\} \notin E$ and $\{y,z\} \in E$. 
    Since $y \notin N(x)$ and $y \in N(z)$, and $N(x)$ and $N(z)$ are nested by Lemma~\ref{lem:nesting}(a-b), it must be the case that $\deg_U(x) \subseteq \deg_U(z)$.  Thus, $x \trianglelefteq z$.
    \item If $z \in U$ and $x, y \in U^c$, then $\deg_U(x) \leq \deg_U(y)$ and $\{y,z\} \notin E$. By Lemma~\ref{lem:nesting}(a-b), $N(x) \subseteq N(y)$. Therefore, $z \notin N(y) \supseteq N(x)$, which means $\{x,z\} \notin E$ and $x \trianglelefteq z$.  
    \item  If $x, y, z \in U^c$, then $\deg_U(x) \leq \deg_U(y) \leq \deg_U(z)$, and $x \trianglelefteq z$ by the transitivity of~$\leq$.
\end{itemize}

\item (Totality) Suppose $x,y \in V$.  If exactly one of $x$ and $y$ belongs to $U$, then $x \trianglelefteq y$ or $x \trianglerighteq y$ since $\{x,y\} \in E$ or $\{x,y\} \notin E$.  If $x,y \in U^c$, then $x \trianglelefteq y$ or $x \trianglerighteq y$ by Lemma~\ref{lem:nesting}(b).  If $x,y \in U$, then $x \preceq y$ or $x \succeq y$ since $\preceq$ is a total order on $U/\sim$, and $x \trianglelefteq y$ or $x \trianglerighteq y$ by Lemma~\ref{lem:nesting}(c-d).
\end{itemize}
Therefore, the binary relation $\trianglelefteq$ is a total order on $V/\sim$.
\end{proof}

We now have all the necessary tools to prove Theorem~\ref{thm:order}.

\begin{proof}[Proof of Theorem~\ref{thm:order}.]
Starting with the forward direction, let $G = (V,E)$ be a special $2$-threshold graph on $n$ vertices.  Then there exists $U \subseteq V$ such that $G$ is $U$-threshold (Remark~\ref{rem:U-threshold}).  We claim that any ordering of $V$ obtained from $\trianglelefteq$ on $V/\sim$ with the elements of each equivalence class inserted consecutively satisfies the desired property. Let $v_1, \ldots, v_n$ be such an ordering and suppose $v_k \in V$ such that $N^<(v_k) \neq \emptyset$.  It suffices to show there is no vertex $v_j$ with $j < k$ such that \begin{center} (i) $v_j \in U^c$ and $\{v_j,v_k\} \in E$ \quad or \quad (ii) $v_j \in U$ and $\{v_j,v_k\} \notin E$. \end{center}  For (i), Lemma~\ref{lem:nesting}(a) implies that $v_k \in U$, which means $v_j \triangleright v_k$ and the ordering $v_1, \ldots, v_n$ could not have arisen as a refinement of $\trianglelefteq$.  For (ii), our assumption that $N^<(v_k) \neq \emptyset$ and (i) imply that $v_k \in U^c$; otherwise, $v_k$ would be a dominating vertex in $G[U]$, which would imply $\{v_j,v_k\} \in E$.  Since $v_k \in U^c$, $v_j \triangleright v_k$ and the ordering $v_1, \ldots, v_n$ could not have arisen as a refinement of $\trianglelefteq$. 

For the reverse direction, suppose a graph $G$ has a construction order $v_1, \ldots, v_n$ for some $U \subseteq V$.  Let $W \subseteq V$ and consider the highest index vertex $v_i$ in $W$.  Then $W \subseteq \{v_1, \ldots, v_i\}$, which means $N_{W}(v_i) = N^{<}(v_i) \cap W.$  Since $N^{<}(v_i) = \emptyset$ or $N^{<}(v_i) = U \cap \{v_1, \ldots, v_{i-1}\}$, it must be the case that $N_{W}(v_i) = \emptyset$ or $N_{W}(v_i) = (W \setminus v_i) \cap U$.  Therefore, $v_i$ is either isolated or $U$-dominating in $G[W]$, which means $G$ is special $2$-threshold. 
\end{proof}

We conclude this section with an extension of Proposition~\ref{prop:ferrersqt} that emerges as a straightforward corollary of Lemma~\ref{lem:nesting}.  While a threshold graph is a $U$-threshold graph where $U$ is its entire vertex set, a Ferrers graph is a connected $U$-threshold graph where $U$ is an independent set.

\begin{corollary}
    A connected $U$-threshold graph is Ferrers if and only if $U$ is independent.
\end{corollary}

\begin{proof}
A Ferrers graph $G = (V,E)$ with $V = R \sqcup C$ is $U$-threshold with $U = C$ by Proposition~\ref{prop:ferrersqt}.  Therefore, it suffices to show that a connected $U$-threshold graph is Ferrers whenever $U$ is independent.
Let $G = (V,E)$ be a $U$-threshold graph such that $U$ is an independent set.  By Lemma~\ref{lem:nesting}(a), $U^c$ is also independent, so $G$ is bipartite with $V = U \sqcup U^c$.  By Lemma~\ref{lem:nesting}(b), the neighborhoods of the vertices in $U^c$ are nested.  Finally, since $U$ is independent, the neighborhoods of the vertices in $U$ are equal to their restrictions to $U^c$, which means they are nested by Lemma~\ref{lem:nesting}(c).  Ordering the vertices in each part, $U$ and $U^c$, by neighborhood containment, one can check that $G$ is Ferrers.
\end{proof}

\section{Weighted Results}\label{sec:weighted}
This is the final section of the article where we present weighted analogues to our main results for edge weighted graphs. An \textbf{edge weighting} of a graph $G = (V,E)$ is a function $\omega: E \rightarrow \mathbb{R}$ that assigns a numerical \textbf{weight} to each edge $e \in E$.  When the vertices of $e$ are specified, i.e., $e = \{v_i,v_j\}$ for some $v_i,v_j \in V$, we abbreviate $\omega(\{v_i,v_j\})$ to $\omega_{i,j}$.  For a given edge weighting $\omega$ of $G$, the weighted analogues of degree, Laplacian matrix, and spanning tree enumerator are defined as follows: \begin{itemize} \item The \textbf{weighted degree} of a vertex $v_i \in V$ is given by
$$\deg(v_i;\omega)=\sum_{v_j \in N(v_i)}\omega_{i,j} = \sum_{\{v_i,v_j\}\in E}\omega_{i,j}.$$
\item The \textbf{weighted Laplacian matrix} $L(G;\omega)$ is the $|V| \times |V|$ matrix whose entries are given by 
  $$
    L(G; \omega)(i,j) =
    \begin{cases}
      \deg(v_i; \omega) & \text{ if } i = j, \\
      - \omega_{i,j} & \text{ if } \{v_i,v_j\} \in E, 
      \\
      0 & \text {otherwise}.
    \end{cases}
  $$
\item The \textbf{weighted spanning tree enumerator} of $G$ is given by
$$\tau(G;\omega) = \sum_{T\in\mathcal{T}(G)}\prod_{e\in E(T)}\omega(e),$$ where $\mathcal{T}(G)$ denotes the set of all spanning trees in $G$.
\end{itemize}
It is straightforward to check that the unweighted analogues of these objects are recovered from the special case where all the weights are equal to $1$. 

Weighted spanning tree enumeration formulae for various classes of graphs are well-established in the literature.  For instance, the weighted version of Cayley's Formula (Theorem~\ref{thm:completeenum}) is known as the Cayley-Pr\"ufer Theorem~\cite{Moon}.

\begin{theorem}[Cayley-Pr\"ufer Theorem]
Let the edges of a complete graph on $n$ vertices be weighted by $\omega_{i,j} = x_ix_j$ for some indeterminates $x_1,\ldots,x_n$. Then $$\tau(K_n;\omega) = \left(\prod_{k=1}^n x_k \right) \left(\sum_{k=1}^n x_k \right)^{n-2}.$$
\end{theorem}

The following weighted version of Theorem~\ref{thm:thresholdenum} for threshold graphs is a special case of a result by Martin and Reiner~\cite[Theorem~4, Equation (11)]{Martin-Reiner}.

\begin{theorem}\label{thm:thresholdweightedenum}
  Let $G = (V,E)$ be a threshold graph with $V = \{v_1,\dots,v_n\}$ and edge weighting $\omega_{i,j}=x_i x_j$ for some indeterminates $x_1,\dots,x_n.$ Then
  \[
   \tau(G;\omega) = \frac{\displaystyle \left(\prod_{k=1}^n x_k\right)
   \left(\prod_{v_i \in D}\left(x_i + \sum_{v_k \in N(v_i)}x_k\right)\right)\left(\prod_{v_j\in I} \sum_{v_k \in N(v_j)} x_k\right)}{\displaystyle \sum_{i=1}^n x_i}
  ,
  \]
  where $D, I \subseteq V$ are the subsets of dominating and isolated vertices in $G$, respectively.
\end{theorem}

Ehrenborg and van Willigenburg~\cite{E-VW} proved a weighted version of Theorem~\ref{thm:ferrersenum} for Ferrers graphs. Their formula makes use of the notion of a conjugate partition:  If $\lambda = (\lambda_1,\lambda_2,\dots,\lambda_m)$ is an integer partition, then its \textbf{conjugate partition} is $\lambda^\prime = (\lambda_1^\prime,\lambda_2^\prime,\dots,\lambda_n^\prime)$ where $n=\lambda_1$ and $\lambda_j^\prime$ counts the number of parts $\lambda_i$ with $\lambda_i \geq j$ for $1 \leq j \leq n$.  For example, the conjugate partition of $(3, 2, 2, 1)$ is $(4, 3, 1)$.  To visualize the relationship between a general partition $\lambda$ and its conjugate $\lambda^\prime$, observe that $\lambda_i$ counts the number of boxes in the $i$th row in the Ferrers diagram of $\lambda$ whereas $\lambda_j^\prime$ counts the number of boxes in the $j$th column.  Alternatively, the Ferrers diagram of $\lambda^\prime$ is obtained by reflecting the Ferrers diagram of $\lambda$ about the diagonal line through its upper leftmost box.

\begin{theorem}\label{thm:ferrersweightedenum}
  Let $G = (V,E)$ be a Ferrers graph with $V = R \sqcup C$, $|R| = m$, $|C| = n$, and edge weighting $\omega_{i,j}=x_i y_j$ for some indeterminates $x_1,\dots , x_m,y_1,\dots,y_n.$ Moreover, let  $\lambda = (\lambda_1,\lambda_2,\dots,\lambda_m)$  be the integer partition corresponding to the Ferrers diagram of $G$. Then
  \[
  \tau(G;\omega) = \left(\prod_{i=1}^m x_i\right)\cdot\left(\prod_{j=1}^n y_j\right)\cdot \left(\prod_{i=2}^m \left(\sum_{k=1}^{\lambda_i}y_k\right)\right) \cdot \left(\prod_{j=2}^n \left(\sum_{k=1}^{\lambda_j^\prime}x_k\right)\right)
  \]
  where $\lambda^\prime = (\lambda_1^\prime,\lambda_2^\prime,\dots,\lambda_n^\prime)$ is the conjugate partition of $\lambda$.
\end{theorem}

Klee and Stamps~\cite{klee_stamps_weighted} presented straightforward proofs for the weighted versions of all the formulae in Section~\ref{sec:enum} using the following weighted analogue to Lemma~\ref{lem:klee-stamps}.
  
\begin{lemma}\label{lem:klee-stamps-2}
Let $G$ be a graph on $n$ vertices with edge weighting $\omega$ and weighted Laplacian matrix $L = L(G;\omega)$, and let $\va= \langle a_i \rangle_{i\in [n]}$ and $\vb= \langle b_i \rangle_{i\in [n]}$ be column vectors in $\mathbb{R}^n$. Then
$$\det\left(L+\va\vb^T\right) = \left(\sum_{i=1}^n a_i\right) \cdot \left(\sum_{i=1}^n b_i\right)\cdot \tau(G; \omega).$$
\end{lemma}

Next, we prove a weighted version of Theorem~\ref{thm:quasiupdate} with which we can apply Lemma~\ref{lem:klee-stamps-2} to prove a weighted spanning tree enumeration formula for special $2$-threshold graphs that generalizes Theorems~\ref{thm:thresholdweightedenum} and \ref{thm:ferrersweightedenum} in the same way Theorem~\ref{thm:quasienum} generalizes Theorems~\ref{thm:thresholdenum} and \ref{thm:ferrersenum}.  

\begin{theorem}\label{thm:weightedquasiupdate}
  Let $G = (V,E)$ be a graph on $n$ vertices whose edges are weighted by $\omega_{i,j} = x_i x_j$ for some indeterminates $x_1,\dots,x_n.$ There exists an upper-triangular rank-one perturbation of the weighted Laplacian matrix of $G$ if and only if $G$ is special $2$-threshold.
\end{theorem}

\begin{proof}
  Starting with the reverse direction, let $G = (V, E)$ be a special $2$-threshold graph on $n$ vertices. By Theorem~\ref{thm:order}, there exists a construction order $v_1,\dots,v_n$ for some $U \subseteq V$.  Let $L = L(G,\omega)$ be the weighted Laplacian matrix of $G$ with respect to that order, and let $\va , \vb \in \mathbb{R}^n$ be defined by
    \[a_i = \begin{cases}x_i & \text{if }v_i\in D,\\ 0 &\text{else},\end{cases}\qquad b_i = \begin{cases}x_i & \text{if }v_i\in U,\\ 0 &\text{else}.\end{cases}\]
  Then $L + \va\vb^T$ is upper triangular since for $1 \leq j < i \leq n$, $(\va\vb^T)(i, j) = x_i x_j$ if and only if $v_i \in D$ and $v_j \in U$ if and only if $\{v_i, v_j\} \in E$ if and only if $L(i, j) = -x_ix_j$.

  For the forward direction, suppose $G = (V, E)$ is a graph with $V = \{v_1, \ldots, v_n\}$ and edge weights $\omega_{i,j} = x_ix_j$ for some indeterminates $x_1, \ldots, x_n$, such that $L = L(G, \omega)$ admits an upper triangular rank-one perturbation, that is, there exists a rank-one matrix $M \in R^{n \times n}$ such that $L + M$ is upper-triangular. Since $M$ is rank-one, it can be factored into an outer product $M = \va\vb^T$ of two column vectors $\va = \langle a_i \rangle$ and $\vb= \langle b_i \rangle$ in $\mathbb{R}^n$. Moreover, since $L + \va\vb^T$ is upper triangular and $L(i,j) \in \{-x_ix_j, 0\}$ for $1 \leq j < i \leq n$, it follows that $M(i,j) \in \{x_ix_j, 0\}$ for $1 \leq j < i \leq n$, and we may assume without loss of generality that $M \in \{0, x_ix_j\}^{n \times n}$ and $a_i, b_i \in \{0, x_i\}$ for $1 \leq i \leq n$. Let $D = \{v_i \in V \mid a_i = x_i\}$ and $U = \{v_i \in V \mid b_i = x_i \}$. If $v_i \in V \setminus D$, then $a_i = 0$, which means $L(i,j) = M(i, j) = 0$ for $1 \leq j < i \leq n$, hence $N^<(v_i) = \emptyset.$ If $v_i \in D$, on the other hand, then $a_i = x_i$, which means
  \[
    M(i,j) = \begin{cases} 
      x_ix_j & \text{if } v_j \in U, \\ 
      0 & \text{else},
    \end{cases} \quad 
    \Longleftrightarrow 
    \quad L(G)(i,j) = \begin{cases} 
      -x_ix_j & \text{if } v_j \in U, \\
      0 & \text{else},
    \end{cases}
  \] 
  hence $N^<(v_i) = U \cap \{v_1, \ldots, v_{i-1}\}.$  Therefore, $v_1, \ldots, v_n$ is a construction order with respect to $U$, and $G$ is special $2$-threshold by Theorem~\ref{thm:order}.
\end{proof}

With this, we can prove the following weighted version of Theorem~\ref{thm:quasienum}.

\begin{theorem}\label{thm:weightednuclear}
  Let $G = (V,E)$ be a $U$-threshold graph with construction order $v_1,\dots,v_n$ and edge weighting $\omega_{i,j} = x_i x_j$ for some indeterminates $x_1, \dots , x_n$. Then 
  $$\tau(G,\omega) = 
  \frac{\displaystyle \left(\prod_{v_i \in V} x_i \right)\cdot \left(\prod_{v_i \in D_U} \left(x_i + \sum_{v_j \in N(v_i)} x_j\right) \right)\cdot \left(\prod_{v_i \in V \setminus D_U} \sum_{v_j \in N(v_i)} x_j\right)} 
  {\displaystyle \left(\sum_{v_i \in D}x_i\right) \cdot \left(\sum_{v_i \in U} x_i\right)},$$
  where $D\subseteq V$ is the set of $U$-dominating vertices in the construction of $G,$ and $D_U = D\cap U$.
\end{theorem}

\begin{proof}
  From the proof of Theorem~\ref{thm:weightedquasiupdate}, the rank-one perturbation $L + \va\vb^T$ where 
  \[a_i = \begin{cases}x_i & \text{if }v_i\in D,\\ 0 &\text{else},\end{cases}\qquad b_i = \begin{cases}x_i & \text{if }v_i\in U,\\ 0 &\text{else},\end{cases}\] is upper triangular. Note that $\va\vb^T(i,i) = 0$ for $v_i \notin D_U$ and $\va\vb^T(i,i) = x_i ^2$ for $v_i \in D_U.$ Thus, the diagonal entries of $L + \va\vb^T$ are given by
  \[\left(L+\va\vb^{T}\right)\left(i,i\right)=
    \begin{cases}
  \sum_{v_j\in N(v_i)}x_i x_j + x_i ^2 & \text{if } v_i \in D_U,\\
  \sum_{v_j\in N(v_i)}x_i x_j & \text{else.}
    \end{cases}
  \]
  Since $L+\va\vb^{T}$ is upper triangular,
  \begin{equation*} \begin{split}
  \det\left(L+\va\vb^{T}\right) &= \prod_{v_i\in D_U} \left(\sum_{v_j\in N(v_i)}x_i x_j + x_i ^2\right) \cdot \prod_{v_i\in V \setminus D_U}\left(\sum_{v_j\in N(v_i)}x_i x_j\right)\\
  &=\left(\prod_{v_i \in V} x_i \right)\cdot \left(\prod_{v_i \in D_U} \left(x_i + \sum_{v_j \in N(v_i)} x_j\right) \right)\cdot \left(\prod_{v_i \in V \setminus D_U} \sum_{v_j \in N(v_i)} x_j\right)
  \end{split}\end{equation*}
  By construction, the entries of $\va$ and $\vb$ sum up to
  $$\sum_{i = 1}^{n} a_i = \sum_{v_i \in D}x_i\quad\text{and}\quad \sum_{i = 1}^{n} b_i = \sum_{v_i \in U}x_i.$$
  The desired expression for $\tau(G,\omega)$ follows directly from Lemma~\ref{lem:klee-stamps-2}.
\end{proof}
\newpage
To see how Theorem~\ref{thm:weightednuclear} provides a common generalization of Theorem~\ref{thm:thresholdweightedenum} and Theorem~\ref{thm:ferrersweightedenum}, observe the following: 
\smallskip
\begin{itemize}
    \item If $G = (V,E)$ is a threshold graph on $n$ vertices, then it is $U$-threshold with $U = V$ by Remark~\ref{remark:threshold-special} and $D = N(v_1)$.  Substituting $D_U = D$, $V \setminus D_U = \{v_1\} \cup I$,  and $$\sum\limits_{v_i \in D} x_i = \deg(v_1,\omega)/x_1 = \sum\limits_{v_j \in N(v_1)} x_j$$ into the formula in Theorem~\ref{thm:weightednuclear} -- and simplifying appropriately -- yields the formula in Theorem~\ref{thm:thresholdweightedenum}.  
    \item If $G = (V,E)$ is a Ferrers graph with $V = R \sqcup C$, then it is $U$-threshold with $U = C$ by Proposition~\ref{prop:ferrersqt}, $D = R$ since every vertex in $R$ is adjacent to $v_1 = c_1$, and $D_U = U \cap D = R \cap C = \emptyset$.  Since $N(r_i) = \{c_j \ | \ 1 \leq j \leq \lambda_i\}$ and $N(c_i) = \{r_j \ | \ 1 \leq j \leq \lambda^\prime_i\}$, it follows that 
    \[
        \sum_{c_j \in N(r_i)} y_j = \sum_{j=1}^{\lambda_i} y_j \quad \text{and}
        \sum_{r_j \in N(c_i)} x_j = \sum_{j=1}^{\lambda'_i} x_j
    \]
    for every $r_i \in R$ and $c_i \in C$. Moreover, $$\sum_{c_i \in C} y_i = \sum_{j = 1}^{m} y_j \quad \text{and} \quad \sum_{r_i \in R} x_i = \sum_{j = 1}^{n} x_j$$ since $C = N(r_1)$, $R = N(c_1)$, $\lambda_1 = m$, and $\lambda^\prime_1 = n$.  Substituting $D = R$, $U = C$, and the above identities into the formula in Theorem~\ref{thm:weightednuclear} -- and simplifying appropriately -- yields the formula in Theorem~\ref{thm:ferrersweightedenum}.
\end{itemize}

\section*{Acknowledgment}

We are grateful to Steven Klee for reading a preliminary draft of this article and offering several valuable suggestions for improving its overall quality and clarity.

\section*{Appendix}\label{sec:appendix}

Here we list out the sets of all spanning trees for the graphs in Figure~\ref{fig:graphex}(b)-(f). 

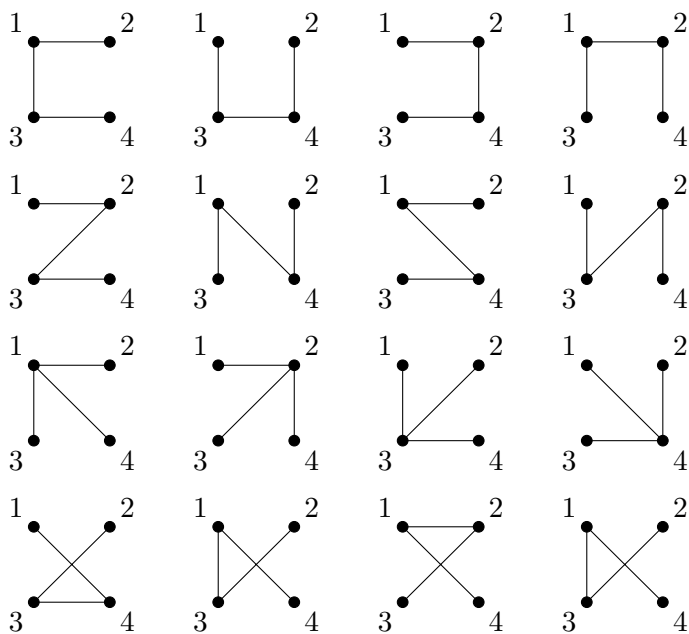
\begin{figure}[H]
    \centering
    \begin{tikzpicture}
    \filldraw (0,0) circle (2pt);
    \node [below left] at (0,0) {$3$};
    \filldraw (0,1) circle (2pt);
    \node [above left] at (0,1) {$1$};
    \filldraw (1,1) circle (2pt);
    \node [above right] at (1,1) {$2$};
    \filldraw (1,0) circle (2pt);
    \node [below right] at (1,0) {$4$};
    \draw (1,1)--(0,1)--(0,0)--(1,0);
    \end{tikzpicture}
    \quad
    \begin{tikzpicture}
    \filldraw (0,0) circle (2pt);
    \node [below left] at (0,0) {$3$};
    \filldraw (0,1) circle (2pt);
    \node [above left] at (0,1) {$1$};
    \filldraw (1,1) circle (2pt);
    \node [above right] at (1,1) {$2$};
    \filldraw (1,0) circle (2pt);
    \node [below right] at (1,0) {$4$};
    \draw (0,1)--(0,0)--(1,0)--(1,1);
    \end{tikzpicture}
    \quad 
    \begin{tikzpicture}
    \filldraw (0,0) circle (2pt);
    \node [below left] at (0,0) {$3$};
    \filldraw (0,1) circle (2pt);
    \node [above left] at (0,1) {$1$};
    \filldraw (1,1) circle (2pt);
    \node [above right] at (1,1) {$2$};
    \filldraw (1,0) circle (2pt);
    \node [below right] at (1,0) {$4$};
    \draw (0,0)--(1,0)--(1,1)--(0,1);
    \end{tikzpicture}
    \quad 
    \begin{tikzpicture}
    \filldraw (0,0) circle (2pt);
    \node [below left] at (0,0) {$3$};
    \filldraw (0,1) circle (2pt);
    \node [above left] at (0,1) {$1$};
    \filldraw (1,1) circle (2pt);
    \node [above right] at (1,1) {$2$};
    \filldraw (1,0) circle (2pt);
    \node [below right] at (1,0) {$4$};
    \draw (1,0)--(1,1)--(0,1)--(0,0);
    \end{tikzpicture}

    \smallskip

    \begin{tikzpicture}
    \filldraw (0,0) circle (2pt);
    \node [below left] at (0,0) {$3$};
    \filldraw (0,1) circle (2pt);
    \node [above left] at (0,1) {$1$};
    \filldraw (1,1) circle (2pt);
    \node [above right] at (1,1) {$2$};
    \filldraw (1,0) circle (2pt);
    \node [below right] at (1,0) {$4$};
    \draw (0,1)--(1,1)--(0,0)--(1,0);
    \end{tikzpicture}
        \quad 
    \begin{tikzpicture}
    \filldraw (0,0) circle (2pt);
    \node [below left] at (0,0) {$3$};
    \filldraw (0,1) circle (2pt);
    \node [above left] at (0,1) {$1$};
    \filldraw (1,1) circle (2pt);
    \node [above right] at (1,1) {$2$};
    \filldraw (1,0) circle (2pt);
    \node [below right] at (1,0) {$4$};
    \draw (0,0)--(0,1)--(1,0)--(1,1);
    \end{tikzpicture}
        \quad 
    \begin{tikzpicture}
    \filldraw (0,0) circle (2pt);
    \node [below left] at (0,0) {$3$};
    \filldraw (0,1) circle (2pt);
    \node [above left] at (0,1) {$1$};
    \filldraw (1,1) circle (2pt);
    \node [above right] at (1,1) {$2$};
    \filldraw (1,0) circle (2pt);
    \node [below right] at (1,0) {$4$};
    \draw (1,1)--(0,1)--(1,0)--(0,0);
    \end{tikzpicture}
        \quad 
    \begin{tikzpicture}
    \filldraw (0,0) circle (2pt);
    \node [below left] at (0,0) {$3$};
    \filldraw (0,1) circle (2pt);
    \node [above left] at (0,1) {$1$};
    \filldraw (1,1) circle (2pt);
    \node [above right] at (1,1) {$2$};
    \filldraw (1,0) circle (2pt);
    \node [below right] at (1,0) {$4$};
    \draw (0,1)--(0,0)--(1,1)--(1,0);
    \end{tikzpicture}
    
    \smallskip
    
    \begin{tikzpicture}
    \filldraw (0,0) circle (2pt);
    \node [below left] at (0,0) {$3$};
    \filldraw (0,1) circle (2pt);
    \node [above left] at (0,1) {$1$};
    \filldraw (1,1) circle (2pt);
    \node [above right] at (1,1) {$2$};
    \filldraw (1,0) circle (2pt);
    \node [below right] at (1,0) {$4$};
    \draw (1,1)--(0,1)--(0,0);
    \draw (0,1)--(1,0);
    \end{tikzpicture}
        \quad 
    \begin{tikzpicture}
    \filldraw (0,0) circle (2pt);
    \node [below left] at (0,0) {$3$};
    \filldraw (0,1) circle (2pt);
    \node [above left] at (0,1) {$1$};
    \filldraw (1,1) circle (2pt);
    \node [above right] at (1,1) {$2$};
    \filldraw (1,0) circle (2pt);
    \node [below right] at (1,0) {$4$};
    \draw (0,1)--(1,1)--(1,0);
    \draw (1,1)--(0,0);
    \end{tikzpicture}
        \quad 
    \begin{tikzpicture}
    \filldraw (0,0) circle (2pt);
    \node [below left] at (0,0) {$3$};
    \filldraw (0,1) circle (2pt);
    \node [above left] at (0,1) {$1$};
    \filldraw (1,1) circle (2pt);
    \node [above right] at (1,1) {$2$};
    \filldraw (1,0) circle (2pt);
    \node [below right] at (1,0) {$4$};
    \draw (0,1)--(0,0)--(1,0);
    \draw (0,0)--(1,1);
    \end{tikzpicture}
        \quad 
    \begin{tikzpicture}
    \filldraw (0,0) circle (2pt);
    \node [below left] at (0,0) {$3$};
    \filldraw (0,1) circle (2pt);
    \node [above left] at (0,1) {$1$};
    \filldraw (1,1) circle (2pt);
    \node [above right] at (1,1) {$2$};
    \filldraw (1,0) circle (2pt);
    \node [below right] at (1,0) {$4$};
    \draw (1,1)--(1,0)--(0,0);
    \draw (1,0)--(0,1);
    \end{tikzpicture}

    \smallskip

    \begin{tikzpicture}
    \filldraw (0,0) circle (2pt);
    \node [below left] at (0,0) {$3$};
    \filldraw (0,1) circle (2pt);
    \node [above left] at (0,1) {$1$};
    \filldraw (1,1) circle (2pt);
    \node [above right] at (1,1) {$2$};
    \filldraw (1,0) circle (2pt);
    \node [below right] at (1,0) {$4$};
    \draw (0,1)--(1,0)--(0,0)--(1,1);
    \end{tikzpicture}
        \quad 
    \begin{tikzpicture}
    \filldraw (0,0) circle (2pt);
    \node [below left] at (0,0) {$3$};
    \filldraw (0,1) circle (2pt);
    \node [above left] at (0,1) {$1$};
    \filldraw (1,1) circle (2pt);
    \node [above right] at (1,1) {$2$};
    \filldraw (1,0) circle (2pt);
    \node [below right] at (1,0) {$4$};
    \draw (1,1)--(0,0)--(0,1)--(1,0);
    \end{tikzpicture}
        \quad 
    \begin{tikzpicture}
    \filldraw (0,0) circle (2pt);
    \node [below left] at (0,0) {$3$};
    \filldraw (0,1) circle (2pt);
    \node [above left] at (0,1) {$1$};
    \filldraw (1,1) circle (2pt);
    \node [above right] at (1,1) {$2$};
    \filldraw (1,0) circle (2pt);
    \node [below right] at (1,0) {$4$};
    \draw (0,0)--(1,1)--(0,1)--(1,0);
    \end{tikzpicture}
        \quad 
    \begin{tikzpicture}
    \filldraw (0,0) circle (2pt);
    \node [below left] at (0,0) {$3$};
    \filldraw (0,1) circle (2pt);
    \node [above left] at (0,1) {$1$};
    \filldraw (1,1) circle (2pt);
    \node [above right] at (1,1) {$2$};
    \filldraw (1,0) circle (2pt);
    \node [below right] at (1,0) {$4$};
    \draw (1,1)--(0,0)--(0,1)--(1,0);
    \end{tikzpicture}
    \caption{All spanning trees of the complete graph in Figure~\ref{fig:graphex}(b).}
    \label{fig:completeex}
\end{figure}

\phantom{space}

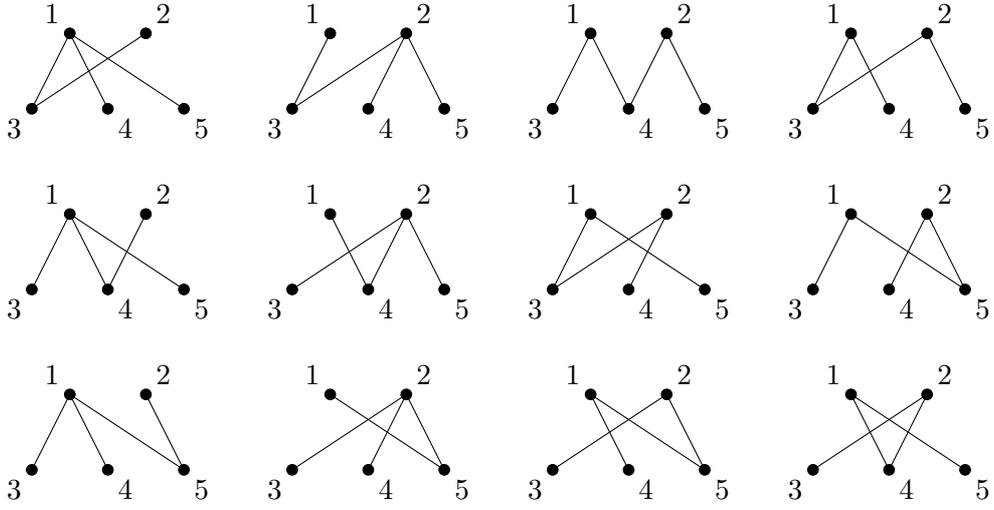
\begin{figure}[H]
    \centering
    \begin{tikzpicture}
    \filldraw (0,1) circle (2pt);
    \node [above left] at (0,1) {$1$};
    \filldraw (1,1) circle (2pt);
    \node [above right] at (1,1) {$2$};
    \filldraw (-0.5,0) circle (2pt);
    \node [below left] at (-0.5,0) {$3$};
    \filldraw (0.5,0) circle (2pt);
    \node [below right] at (0.5,0) {$4$};
    \filldraw (1.5,0) circle (2pt);
    \node [below right] at (1.5,0) {$5$};
    \draw (0,1)--(-0.5,0);
    \draw (0,1)--(0.5,0);
    \draw (0,1)--(1.5,0);
    \draw (1,1)--(-0.5,0);
    \end{tikzpicture} \quad
    \begin{tikzpicture}
    \filldraw (0,1) circle (2pt);
    \node [above left] at (0,1) {$1$};
    \filldraw (1,1) circle (2pt);
    \node [above right] at (1,1) {$2$};
    \filldraw (-0.5,0) circle (2pt);
    \node [below left] at (-0.5,0) {$3$};
    \filldraw (0.5,0) circle (2pt);
    \node [below right] at (0.5,0) {$4$};
    \filldraw (1.5,0) circle (2pt);
    \node [below right] at (1.5,0) {$5$};
    \draw (0,1)--(-0.5,0);
    \draw (1,1)--(-0.5,0);
    \draw (1,1)--(0.5,0);
    \draw (1,1)--(1.5,0);
    \end{tikzpicture} \quad
        \begin{tikzpicture}
    \filldraw (0,1) circle (2pt);
    \node [above left] at (0,1) {$1$};
    \filldraw (1,1) circle (2pt);
    \node [above right] at (1,1) {$2$};
    \filldraw (-0.5,0) circle (2pt);
    \node [below left] at (-0.5,0) {$3$};
    \filldraw (0.5,0) circle (2pt);
    \node [below right] at (0.5,0) {$4$};
    \filldraw (1.5,0) circle (2pt);
    \node [below right] at (1.5,0) {$5$};
    \draw (0,1)--(-0.5,0);
    \draw (0,1)--(0.5,0);
    \draw (1,1)--(0.5,0);
    \draw (1,1)--(1.5,0);
    \end{tikzpicture} \quad
        \begin{tikzpicture}
    \filldraw (0,1) circle (2pt);
    \node [above left] at (0,1) {$1$};
    \filldraw (1,1) circle (2pt);
    \node [above right] at (1,1) {$2$};
    \filldraw (-0.5,0) circle (2pt);
    \node [below left] at (-0.5,0) {$3$};
    \filldraw (0.5,0) circle (2pt);
    \node [below right] at (0.5,0) {$4$};
    \filldraw (1.5,0) circle (2pt);
    \node [below right] at (1.5,0) {$5$};
    \draw (0,1)--(-0.5,0);
    \draw (0,1)--(0.5,0);
    \draw (1,1)--(-0.5,0);
    \draw (1,1)--(1.5,0);
    \end{tikzpicture}
    
    \bigskip
    
        \begin{tikzpicture}
    \filldraw (0,1) circle (2pt);
    \node [above left] at (0,1) {$1$};
    \filldraw (1,1) circle (2pt);
    \node [above right] at (1,1) {$2$};
    \filldraw (-0.5,0) circle (2pt);
    \node [below left] at (-0.5,0) {$3$};
    \filldraw (0.5,0) circle (2pt);
    \node [below right] at (0.5,0) {$4$};
    \filldraw (1.5,0) circle (2pt);
    \node [below right] at (1.5,0) {$5$};
    \draw (0,1)--(-0.5,0);
    \draw (0,1)--(0.5,0);
    \draw (0,1)--(1.5,0);
    \draw (1,1)--(0.5,0);
    \end{tikzpicture} \quad
    \begin{tikzpicture}
    \filldraw (0,1) circle (2pt);
    \node [above left] at (0,1) {$1$};
    \filldraw (1,1) circle (2pt);
    \node [above right] at (1,1) {$2$};
    \filldraw (-0.5,0) circle (2pt);
    \node [below left] at (-0.5,0) {$3$};
    \filldraw (0.5,0) circle (2pt);
    \node [below right] at (0.5,0) {$4$};
    \filldraw (1.5,0) circle (2pt);
    \node [below right] at (1.5,0) {$5$};
    \draw (0,1)--(0.5,0);
    \draw (1,1)--(-0.5,0);
    \draw (1,1)--(0.5,0);
    \draw (1,1)--(1.5,0);
    \end{tikzpicture} \quad
        \begin{tikzpicture}
    \filldraw (0,1) circle (2pt);
    \node [above left] at (0,1) {$1$};
    \filldraw (1,1) circle (2pt);
    \node [above right] at (1,1) {$2$};
    \filldraw (-0.5,0) circle (2pt);
    \node [below left] at (-0.5,0) {$3$};
    \filldraw (0.5,0) circle (2pt);
    \node [below right] at (0.5,0) {$4$};
    \filldraw (1.5,0) circle (2pt);
    \node [below right] at (1.5,0) {$5$};
    \draw (0,1)--(-0.5,0);
    \draw (0,1)--(1.5,0);
    \draw (1,1)--(-0.5,0);
    \draw (1,1)--(0.5,0);
    \end{tikzpicture} \quad
        \begin{tikzpicture}
    \filldraw (0,1) circle (2pt);
    \node [above left] at (0,1) {$1$};
    \filldraw (1,1) circle (2pt);
    \node [above right] at (1,1) {$2$};
    \filldraw (-0.5,0) circle (2pt);
    \node [below left] at (-0.5,0) {$3$};
    \filldraw (0.5,0) circle (2pt);
    \node [below right] at (0.5,0) {$4$};
    \filldraw (1.5,0) circle (2pt);
    \node [below right] at (1.5,0) {$5$};
    \draw (0,1)--(-0.5,0);
    \draw (0,1)--(1.5,0);
    \draw (1,1)--(0.5,0);
    \draw (1,1)--(1.5,0);
    \end{tikzpicture}
    
    \bigskip
    
        \begin{tikzpicture}
    \filldraw (0,1) circle (2pt);
    \node [above left] at (0,1) {$1$};
    \filldraw (1,1) circle (2pt);
    \node [above right] at (1,1) {$2$};
    \filldraw (-0.5,0) circle (2pt);
    \node [below left] at (-0.5,0) {$3$};
    \filldraw (0.5,0) circle (2pt);
    \node [below right] at (0.5,0) {$4$};
    \filldraw (1.5,0) circle (2pt);
    \node [below right] at (1.5,0) {$5$};
    \draw (0,1)--(-0.5,0);
    \draw (0,1)--(0.5,0);
    \draw (0,1)--(1.5,0);
    \draw (1,1)--(1.5,0);
    \end{tikzpicture} \quad
    \begin{tikzpicture}
    \filldraw (0,1) circle (2pt);
    \node [above left] at (0,1) {$1$};
    \filldraw (1,1) circle (2pt);
    \node [above right] at (1,1) {$2$};
    \filldraw (-0.5,0) circle (2pt);
    \node [below left] at (-0.5,0) {$3$};
    \filldraw (0.5,0) circle (2pt);
    \node [below right] at (0.5,0) {$4$};
    \filldraw (1.5,0) circle (2pt);
    \node [below right] at (1.5,0) {$5$};
    \draw (0,1)--(1.5,0);
    \draw (1,1)--(-0.5,0);
    \draw (1,1)--(0.5,0);
    \draw (1,1)--(1.5,0);
    \end{tikzpicture} \quad
        \begin{tikzpicture}
    \filldraw (0,1) circle (2pt);
    \node [above left] at (0,1) {$1$};
    \filldraw (1,1) circle (2pt);
    \node [above right] at (1,1) {$2$};
    \filldraw (-0.5,0) circle (2pt);
    \node [below left] at (-0.5,0) {$3$};
    \filldraw (0.5,0) circle (2pt);
    \node [below right] at (0.5,0) {$4$};
    \filldraw (1.5,0) circle (2pt);
    \node [below right] at (1.5,0) {$5$};
    \draw (0,1)--(0.5,0);
    \draw (0,1)--(1.5,0);
    \draw (1,1)--(-0.5,0);
    \draw (1,1)--(1.5,0);
    \end{tikzpicture} \quad
        \begin{tikzpicture}
    \filldraw (0,1) circle (2pt);
    \node [above left] at (0,1) {$1$};
    \filldraw (1,1) circle (2pt);
    \node [above right] at (1,1) {$2$};
    \filldraw (-0.5,0) circle (2pt);
    \node [below left] at (-0.5,0) {$3$};
    \filldraw (0.5,0) circle (2pt);
    \node [below right] at (0.5,0) {$4$};
    \filldraw (1.5,0) circle (2pt);
    \node [below right] at (1.5,0) {$5$};
    \draw (0,1)--(0.5,0);
    \draw (0,1)--(1.5,0);
    \draw (1,1)--(-0.5,0);
    \draw (1,1)--(0.5,0);
    \end{tikzpicture}
    \caption{All spanning trees of the complete bipartite graph in Figure~\ref{fig:graphex}(c).}
    \label{fig:bipartiteex}
\end{figure}

\phantom{space}

\begin{figure}[H]
    \centering
    
    \begin{tikzpicture}
        \filldraw (0,0) circle (2pt);
        \node [below left] at (0,0) {$3$};
        \filldraw (1,0) circle (2pt);
        \node [below] at (1,0) {$4$};
        \filldraw (1,1) circle (2pt);
        \node [above left] at (1,1) {$1$};
        \filldraw (2,1) circle (2pt);
        \node [above right] at (2,1) {$2$};
        \filldraw (2,0) circle (2pt);
        \node [below right] at (2,0) {$5$};
        \draw (0,0)--(1,0);
        \draw (1,0)--(2,0)--(2,1)--(1,1);
    \end{tikzpicture} \quad
    \begin{tikzpicture}
        \filldraw (0,0) circle (2pt);
        \node [below left] at (0,0) {$3$};
        \filldraw (1,0) circle (2pt);
        \node [below] at (1,0) {$4$};
        \filldraw (1,1) circle (2pt);
        \node [above left] at (1,1) {$1$};
        \filldraw (2,1) circle (2pt);
        \node [above right] at (2,1) {$2$};
        \filldraw (2,0) circle (2pt);
        \node [below right] at (2,0) {$5$};
        \draw (0,0)--(1,0);
        \draw (2,0)--(2,1)--(1,1)--(1,0);
    \end{tikzpicture} \quad
    \begin{tikzpicture}
        \filldraw (0,0) circle (2pt);
        \node [below left] at (0,0) {$3$};
        \filldraw (1,0) circle (2pt);
        \node [below] at (1,0) {$4$};
        \filldraw (1,1) circle (2pt);
        \node [above left] at (1,1) {$1$};
        \filldraw (2,1) circle (2pt);
        \node [above right] at (2,1) {$2$};
        \filldraw (2,0) circle (2pt);
        \node [below right] at (2,0) {$5$};
        \draw (0,0)--(1,0);
        \draw (1,1)--(2,1)--(1,0)--(2,0);
    \end{tikzpicture} \quad
    \begin{tikzpicture}
         \filldraw (0,0) circle (2pt);
        \node [below left] at (0,0) {$3$};
        \filldraw (1,0) circle (2pt);
        \node [below] at (1,0) {$4$};
        \filldraw (1,1) circle (2pt);
        \node [above left] at (1,1) {$1$};
        \filldraw (2,1) circle (2pt);
        \node [above right] at (2,1) {$2$};
        \filldraw (2,0) circle (2pt);
        \node [below right] at (2,0) {$5$};
        \draw (0,0)--(1,0);
        \draw (1,1)--(1,0)--(2,1)--(2,0);
    \end{tikzpicture} 
    
    \bigskip

    \begin{tikzpicture}
         \filldraw (0,0) circle (2pt);
        \node [below left] at (0,0) {$3$};
        \filldraw (1,0) circle (2pt);
        \node [below] at (1,0) {$4$};
        \filldraw (1,1) circle (2pt);
        \node [above left] at (1,1) {$1$};
        \filldraw (2,1) circle (2pt);
        \node [above right] at (2,1) {$2$};
        \filldraw (2,0) circle (2pt);
        \node [below right] at (2,0) {$5$};
        \draw (0,0)--(1,0);
        \draw (2,1)--(1,1)--(1,0)--(2,0);
    \end{tikzpicture} \quad
    \begin{tikzpicture}
        \filldraw (0,0) circle (2pt);
        \node [below left] at (0,0) {$3$};
        \filldraw (1,0) circle (2pt);
        \node [below] at (1,0) {$4$};
        \filldraw (1,1) circle (2pt);
        \node [above left] at (1,1) {$1$};
        \filldraw (2,1) circle (2pt);
        \node [above right] at (2,1) {$2$};
        \filldraw (2,0) circle (2pt);
        \node [below right] at (2,0) {$5$};
        \draw (0,0)--(1,0);
        \draw (1,1)--(1,0)--(2,0)--(2,1);
    \end{tikzpicture} \quad
    \begin{tikzpicture}
        \filldraw (0,0) circle (2pt);
        \node [below left] at (0,0) {$3$};
        \filldraw (1,0) circle (2pt);
        \node [below] at (1,0) {$4$};
        \filldraw (1,1) circle (2pt);
        \node [above left] at (1,1) {$1$};
        \filldraw (2,1) circle (2pt);
        \node [above right] at (2,1) {$2$};
        \filldraw (2,0) circle (2pt);
        \node [below right] at (2,0) {$5$};
        \draw (0,0)--(1,0)--(2,0);
        \draw (2,1)--(1,0)--(1,1);
    \end{tikzpicture} \quad
    \begin{tikzpicture}
        \filldraw (0,0) circle (2pt);
        \node [below left] at (0,0) {$3$};
        \filldraw (1,0) circle (2pt);
        \node [below] at (1,0) {$4$};
        \filldraw (1,1) circle (2pt);
        \node [above left] at (1,1) {$1$};
        \filldraw (2,1) circle (2pt);
        \node [above right] at (2,1) {$2$};
        \filldraw (2,0) circle (2pt);
        \node [below right] at (2,0) {$5$};
        \draw (0,0)--(1,0)--(2,1);
        \draw (1,1)--(2,1)--(2,0);
    \end{tikzpicture}
    \caption{All spanning trees of the threshold graph in Figure~\ref{fig:graphex}(d).}
    \label{fig:thresholdex}
\end{figure}
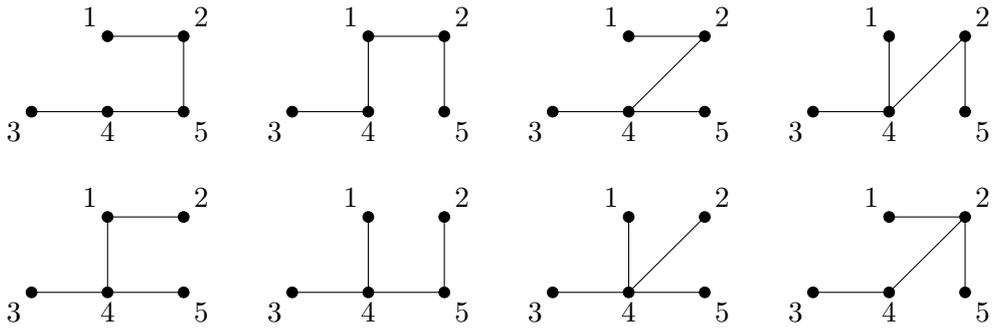

\phantom{space}

\begin{figure}[H]
    \centering
    \begin{tikzpicture}
        \filldraw (0,0) circle (2pt);
        \node [below left] at (0,0) {$3$};
        \filldraw (1,0) circle (2pt);
        \node [below] at (1,0) {$4$};
        \filldraw (1,1) circle (2pt);
        \node [above left] at (1,1) {$1$};
        \filldraw (2,1) circle (2pt);
        \node [above right] at (2,1) {$2$};
        \filldraw (2,0) circle (2pt);
        \node [below right] at (2,0) {$5$};
        \draw (0,0)--(1,0);
        \draw (1,0)--(2,0)--(2,1)--(1,1);
    \end{tikzpicture} \quad
        \begin{tikzpicture}
        \filldraw (0,0) circle (2pt);
        \node [below left] at (0,0) {$3$};
        \filldraw (1,0) circle (2pt);
        \node [below] at (1,0) {$4$};
        \filldraw (1,1) circle (2pt);
        \node [above left] at (1,1) {$1$};
        \filldraw (2,1) circle (2pt);
        \node [above right] at (2,1) {$2$};
        \filldraw (2,0) circle (2pt);
        \node [below right] at (2,0) {$5$};
        \draw (0,0)--(1,0);
        \draw (2,0)--(2,1)--(1,1)--(1,0);
    \end{tikzpicture} \quad
    \begin{tikzpicture}
        \filldraw (0,0) circle (2pt);
        \node [below left] at (0,0) {$3$};
        \filldraw (1,0) circle (2pt);
        \node [below] at (1,0) {$4$};
        \filldraw (1,1) circle (2pt);
        \node [above left] at (1,1) {$1$};
        \filldraw (2,1) circle (2pt);
        \node [above right] at (2,1) {$2$};
        \filldraw (2,0) circle (2pt);
        \node [below right] at (2,0) {$5$};
        \draw (0,0)--(1,0);
        \draw (1,0)--(2,0)--(1,1)--(2,1);
    \end{tikzpicture} \quad
    \begin{tikzpicture}
        \filldraw (0,0) circle (2pt);
        \node [below left] at (0,0) {$3$};
        \filldraw (1,0) circle (2pt);
        \node [below] at (1,0) {$4$};
        \filldraw (1,1) circle (2pt);
        \node [above left] at (1,1) {$1$};
        \filldraw (2,1) circle (2pt);
        \node [above right] at (2,1) {$2$};
        \filldraw (2,0) circle (2pt);
        \node [below right] at (2,0) {$5$};
        \draw (0,0)--(1,0);
        \draw (1,0)--(1,1)--(2,0)--(2,1);
    \end{tikzpicture}
    
    \bigskip

    \begin{tikzpicture}
        \filldraw (0,0) circle (2pt);
        \node [below left] at (0,0) {$3$};
        \filldraw (1,0) circle (2pt);
        \node [below] at (1,0) {$4$};
        \filldraw (1,1) circle (2pt);
        \node [above left] at (1,1) {$1$};
        \filldraw (2,1) circle (2pt);
        \node [above right] at (2,1) {$2$};
        \filldraw (2,0) circle (2pt);
        \node [below right] at (2,0) {$5$};
        \draw (0,0)--(1,0);
        \draw (2,1)--(1,1)--(1,0)--(2,0);
    \end{tikzpicture} \quad 
    \begin{tikzpicture}
        \filldraw (0,0) circle (2pt);
        \node [below left] at (0,0) {$3$};
        \filldraw (1,0) circle (2pt);
        \node [below] at (1,0) {$4$};
        \filldraw (1,1) circle (2pt);
        \node [above left] at (1,1) {$1$};
        \filldraw (2,1) circle (2pt);
        \node [above right] at (2,1) {$2$};
        \filldraw (2,0) circle (2pt);
        \node [below right] at (2,0) {$5$};
        \draw (0,0)--(1,0);
        \draw (1,1)--(1,0)--(2,0)--(2,1);
    \end{tikzpicture} \quad
    \begin{tikzpicture}
        \filldraw (0,0) circle (2pt);
        \node [below left] at (0,0) {$3$};
        \filldraw (1,0) circle (2pt);
        \node [below] at (1,0) {$4$};
        \filldraw (1,1) circle (2pt);
        \node [above left] at (1,1) {$1$};
        \filldraw (2,1) circle (2pt);
        \node [above right] at (2,1) {$2$};
        \filldraw (2,0) circle (2pt);
        \node [below right] at (2,0) {$5$};
        \draw (0,0)--(1,0)--(2,0);
        \draw (1,1)--(2,0)--(2,1);
    \end{tikzpicture} \quad
    \begin{tikzpicture}
        \filldraw (0,0) circle (2pt);
        \node [below left] at (0,0) {$3$};
        \filldraw (1,0) circle (2pt);
        \node [below] at (1,0) {$4$};
        \filldraw (1,1) circle (2pt);
        \node [above left] at (1,1) {$1$};
        \filldraw (2,1) circle (2pt);
        \node [above right] at (2,1) {$2$};
        \filldraw (2,0) circle (2pt);
        \node [below right] at (2,0) {$5$};
        \draw (0,0)--(1,0)--(1,1);
        \draw (2,1)--(1,1)--(2,0);
    \end{tikzpicture}
    \caption{All spanning trees of special $2$-threshold graph in Figure~\ref{fig:graphex}(e).}
    \label{fig:quasienum}
\end{figure}
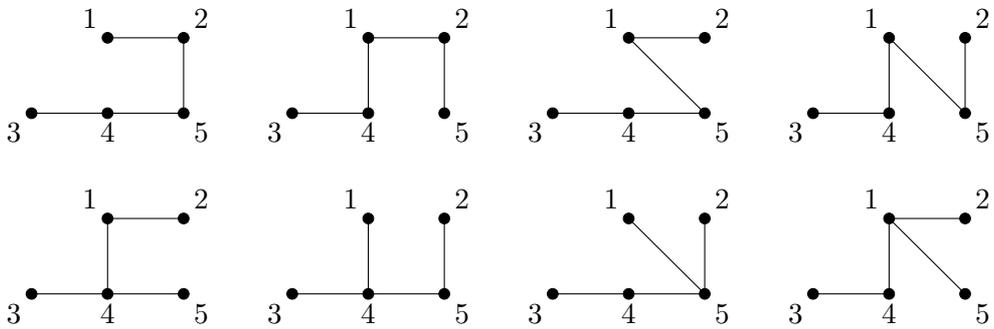

\phantom{space}

\newpage
\phantom{space}

\begin{figure}[H]
    \centering
     \begin{tikzpicture}
        \filldraw (0,0) circle (2pt);
        \node[below] at (0,0) {$4$};
        \filldraw (1,0) circle (2pt);
        \node[below] at (1,0) {$5$};
        \filldraw (2,0) circle (2pt);
        \node[below] at (2,0) {$6$};
        \filldraw (3,0) circle (2pt);
        \node[below] at (3,0) {$7$};
        \filldraw (0.5,1) circle (2pt);
        \node[above] at (0.5,1) {$1$};
        \filldraw (1.5,1) circle (2pt);
        \node[above] at (1.5,1) {$2$};
        \filldraw (2.5,1) circle (2pt);
        \node[above] at (2.5,1) {$3$};
        \draw (0.5,1) -- (2,0);
        \draw (0.5,1) -- (3,0);
        \draw (1.5,1) -- (0,0);
        \draw (1.5,1) -- (1,0);
        \draw (1.5,1) -- (2,0);
        \draw (2.5,1) -- (0,0);
    \end{tikzpicture}
    \quad
    \begin{tikzpicture}
        \filldraw (0,0) circle (2pt);
        \node[below] at (0,0) {$4$};
        \filldraw (1,0) circle (2pt);
        \node[below] at (1,0) {$5$};
        \filldraw (2,0) circle (2pt);
        \node[below] at (2,0) {$6$};
        \filldraw (3,0) circle (2pt);
        \node[below] at (3,0) {$7$};
        \filldraw (0.5,1) circle (2pt);
        \node[above] at (0.5,1) {$1$};
        \filldraw (1.5,1) circle (2pt);
        \node[above] at (1.5,1) {$2$};
        \filldraw (2.5,1) circle (2pt);
        \node[above] at (2.5,1) {$3$};
        \draw (0.5,1) -- (1,0);
        \draw (0.5,1) -- (3,0);
        \draw (1.5,1) -- (0,0);
        \draw (1.5,1) -- (1,0);
        \draw (1.5,1) -- (2,0);
        \draw (2.5,1) -- (0,0);
    \end{tikzpicture}
    \quad
    \begin{tikzpicture}
        \filldraw (0,0) circle (2pt);
        \node[below] at (0,0) {$4$};
        \filldraw (1,0) circle (2pt);
        \node[below] at (1,0) {$5$};
        \filldraw (2,0) circle (2pt);
        \node[below] at (2,0) {$6$};
        \filldraw (3,0) circle (2pt);
        \node[below] at (3,0) {$7$};
        \filldraw (0.5,1) circle (2pt);
        \node[above] at (0.5,1) {$1$};
        \filldraw (1.5,1) circle (2pt);
        \node[above] at (1.5,1) {$2$};
        \filldraw (2.5,1) circle (2pt);
        \node[above] at (2.5,1) {$3$};
        \draw (0.5,1) -- (1,0);
        \draw (0.5,1) -- (2,0);
        \draw (0.5,1) -- (3,0);
        \draw (1.5,1) -- (0,0);
        \draw (1.5,1) -- (2,0);
        \draw (2.5,1) -- (0,0);
    \end{tikzpicture}
    \quad
    \begin{tikzpicture}
        \filldraw (0,0) circle (2pt);
        \node[below] at (0,0) {$4$};
        \filldraw (1,0) circle (2pt);
        \node[below] at (1,0) {$5$};
        \filldraw (2,0) circle (2pt);
        \node[below] at (2,0) {$6$};
        \filldraw (3,0) circle (2pt);
        \node[below] at (3,0) {$7$};
        \filldraw (0.5,1) circle (2pt);
        \node[above] at (0.5,1) {$1$};
        \filldraw (1.5,1) circle (2pt);
        \node[above] at (1.5,1) {$2$};
        \filldraw (2.5,1) circle (2pt);
        \node[above] at (2.5,1) {$3$};
        \draw (0.5,1) -- (1,0);
        \draw (0.5,1) -- (2,0);
        \draw (0.5,1) -- (3,0);
        \draw (1.5,1) -- (0,0);
        \draw (1.5,1) -- (1,0);
        \draw (2.5,1) -- (0,0);
    \end{tikzpicture}
    
    \bigskip
    
    \begin{tikzpicture}
        \filldraw (0,0) circle (2pt);
        \node[below] at (0,0) {$4$};
        \filldraw (1,0) circle (2pt);
        \node[below] at (1,0) {$5$};
        \filldraw (2,0) circle (2pt);
        \node[below] at (2,0) {$6$};
        \filldraw (3,0) circle (2pt);
        \node[below] at (3,0) {$7$};
        \filldraw (0.5,1) circle (2pt);
        \node[above] at (0.5,1) {$1$};
        \filldraw (1.5,1) circle (2pt);
        \node[above] at (1.5,1) {$2$};
        \filldraw (2.5,1) circle (2pt);
        \node[above] at (2.5,1) {$3$};
        \draw (0.5,1) -- (0,0);
        \draw (0.5,1) -- (3,0);
        \draw (1.5,1) -- (0,0);
        \draw (1.5,1) -- (1,0);
        \draw (1.5,1) -- (2,0);
        \draw (2.5,1) -- (0,0);
    \end{tikzpicture}
    \quad
    \begin{tikzpicture}
        \filldraw (0,0) circle (2pt);
        \node[below] at (0,0) {$4$};
        \filldraw (1,0) circle (2pt);
        \node[below] at (1,0) {$5$};
        \filldraw (2,0) circle (2pt);
        \node[below] at (2,0) {$6$};
        \filldraw (3,0) circle (2pt);
        \node[below] at (3,0) {$7$};
        \filldraw (0.5,1) circle (2pt);
        \node[above] at (0.5,1) {$1$};
        \filldraw (1.5,1) circle (2pt);
        \node[above] at (1.5,1) {$2$};
        \filldraw (2.5,1) circle (2pt);
        \node[above] at (2.5,1) {$3$};
        \draw (0.5,1) -- (0,0);
        \draw (0.5,1) -- (2,0);
        \draw (0.5,1) -- (3,0);
        \draw (1.5,1) -- (1,0);
        \draw (1.5,1) -- (2,0);
        \draw (2.5,1) -- (0,0);
    \end{tikzpicture}
    \quad
    \begin{tikzpicture}
        \filldraw (0,0) circle (2pt);
        \node[below] at (0,0) {$4$};
        \filldraw (1,0) circle (2pt);
        \node[below] at (1,0) {$5$};
        \filldraw (2,0) circle (2pt);
        \node[below] at (2,0) {$6$};
        \filldraw (3,0) circle (2pt);
        \node[below] at (3,0) {$7$};
        \filldraw (0.5,1) circle (2pt);
        \node[above] at (0.5,1) {$1$};
        \filldraw (1.5,1) circle (2pt);
        \node[above] at (1.5,1) {$2$};
        \filldraw (2.5,1) circle (2pt);
        \node[above] at (2.5,1) {$3$};
        \draw (0.5,1) -- (0,0);
        \draw (0.5,1) -- (2,0);
        \draw (0.5,1) -- (3,0);
        \draw (1.5,1) -- (0,0);
        \draw (1.5,1) -- (1,0);
        \draw (2.5,1) -- (0,0);
    \end{tikzpicture}
    \quad
    \begin{tikzpicture}
        \filldraw (0,0) circle (2pt);
        \node[below] at (0,0) {$4$};
        \filldraw (1,0) circle (2pt);
        \node[below] at (1,0) {$5$};
        \filldraw (2,0) circle (2pt);
        \node[below] at (2,0) {$6$};
        \filldraw (3,0) circle (2pt);
        \node[below] at (3,0) {$7$};
        \filldraw (0.5,1) circle (2pt);
        \node[above] at (0.5,1) {$1$};
        \filldraw (1.5,1) circle (2pt);
        \node[above] at (1.5,1) {$2$};
        \filldraw (2.5,1) circle (2pt);
        \node[above] at (2.5,1) {$3$};
        \draw (0.5,1) -- (0,0);
        \draw (0.5,1) -- (1,0);
        \draw (0.5,1) -- (3,0);
        \draw (1.5,1) -- (1,0);
        \draw (1.5,1) -- (2,0);
        \draw (2.5,1) -- (0,0);
    \end{tikzpicture}
    
    \bigskip
    
    \begin{tikzpicture}
        \filldraw (0,0) circle (2pt);
        \node[below] at (0,0) {$4$};
        \filldraw (1,0) circle (2pt);
        \node[below] at (1,0) {$5$};
        \filldraw (2,0) circle (2pt);
        \node[below] at (2,0) {$6$};
        \filldraw (3,0) circle (2pt);
        \node[below] at (3,0) {$7$};
        \filldraw (0.5,1) circle (2pt);
        \node[above] at (0.5,1) {$1$};
        \filldraw (1.5,1) circle (2pt);
        \node[above] at (1.5,1) {$2$};
        \filldraw (2.5,1) circle (2pt);
        \node[above] at (2.5,1) {$3$};
        \draw (0.5,1) -- (0,0);
        \draw (0.5,1) -- (1,0);
        \draw (0.5,1) -- (3,0);
        \draw (1.5,1) -- (0,0);
        \draw (1.5,1) -- (2,0);
        \draw (2.5,1) -- (0,0);
    \end{tikzpicture}
    \quad
    \begin{tikzpicture}
        \filldraw (0,0) circle (2pt);
        \node[below] at (0,0) {$4$};
        \filldraw (1,0) circle (2pt);
        \node[below] at (1,0) {$5$};
        \filldraw (2,0) circle (2pt);
        \node[below] at (2,0) {$6$};
        \filldraw (3,0) circle (2pt);
        \node[below] at (3,0) {$7$};
        \filldraw (0.5,1) circle (2pt);
        \node[above] at (0.5,1) {$1$};
        \filldraw (1.5,1) circle (2pt);
        \node[above] at (1.5,1) {$2$};
        \filldraw (2.5,1) circle (2pt);
        \node[above] at (2.5,1) {$3$};
        \draw (0.5,1) -- (0,0);
        \draw (0.5,1) -- (1,0);
        \draw (0.5,1) -- (2,0);
        \draw (0.5,1) -- (3,0);
        \draw (1.5,1) -- (2,0);
        \draw (2.5,1) -- (0,0);
    \end{tikzpicture}
    \quad
    \begin{tikzpicture}
        \filldraw (0,0) circle (2pt);
        \node[below] at (0,0) {$4$};
        \filldraw (1,0) circle (2pt);
        \node[below] at (1,0) {$5$};
        \filldraw (2,0) circle (2pt);
        \node[below] at (2,0) {$6$};
        \filldraw (3,0) circle (2pt);
        \node[below] at (3,0) {$7$};
        \filldraw (0.5,1) circle (2pt);
        \node[above] at (0.5,1) {$1$};
        \filldraw (1.5,1) circle (2pt);
        \node[above] at (1.5,1) {$2$};
        \filldraw (2.5,1) circle (2pt);
        \node[above] at (2.5,1) {$3$};
        \draw (0.5,1) -- (0,0);
        \draw (0.5,1) -- (1,0);
        \draw (0.5,1) -- (2,0);
        \draw (0.5,1) -- (3,0);
        \draw (1.5,1) -- (1,0);
        \draw (2.5,1) -- (0,0);
    \end{tikzpicture}
    \quad
    \begin{tikzpicture}
        \filldraw (0,0) circle (2pt);
        \node[below] at (0,0) {$4$};
        \filldraw (1,0) circle (2pt);
        \node[below] at (1,0) {$5$};
        \filldraw (2,0) circle (2pt);
        \node[below] at (2,0) {$6$};
        \filldraw (3,0) circle (2pt);
        \node[below] at (3,0) {$7$};
        \filldraw (0.5,1) circle (2pt);
        \node[above] at (0.5,1) {$1$};
        \filldraw (1.5,1) circle (2pt);
        \node[above] at (1.5,1) {$2$};
        \filldraw (2.5,1) circle (2pt);
        \node[above] at (2.5,1) {$3$};
        \draw (0.5,1) -- (0,0);
        \draw (0.5,1) -- (1,0);
        \draw (0.5,1) -- (2,0);
        \draw (0.5,1) -- (3,0);
        \draw (1.5,1) -- (0,0);
        \draw (2.5,1) -- (0,0);
    \end{tikzpicture}
    \caption{All spanning trees of the Ferrers graph in Figure~\ref{fig:graphex}(f).}
    \label{fig:ferrersex}
\end{figure}
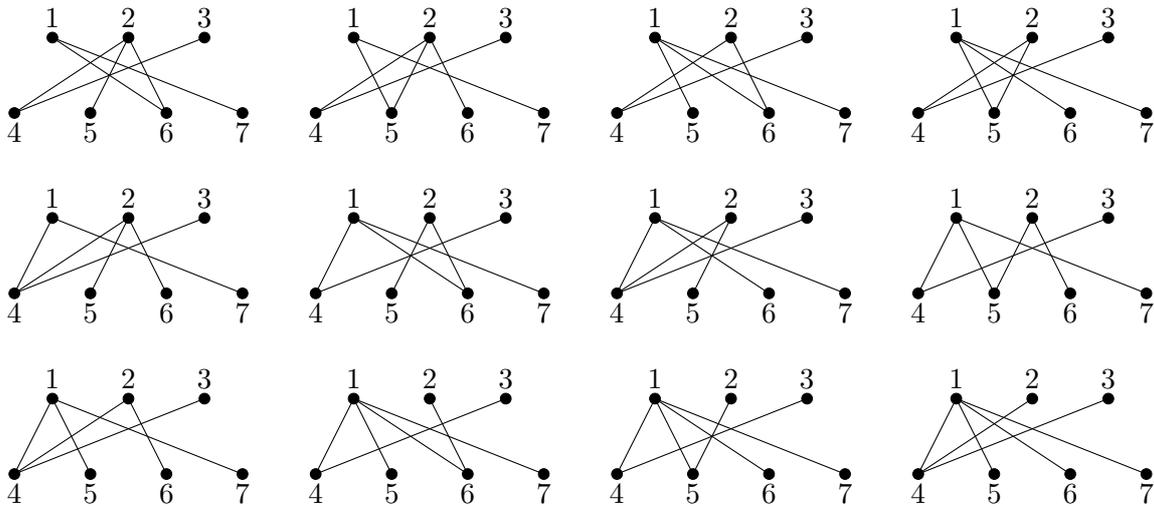

\phantom{space}

\bibliographystyle{abbrv}
\bibliography{biblio}

\begin{thebibliography}{10}

\bibitem{Aigner-Ziegler}
M.~Aigner and G.~M. Ziegler.
\newblock {\em Proofs from {THE BOOK}}.
\newblock Springer-Verlag Berlin Heidelberg, 6th edition, 2018.

\bibitem{atajan_inaba_04}
T.~Atajan and H.~Inaba.
\newblock Network reliability analysis by counting the number of spanning
  trees.
\newblock In {\em {IEEE} {I}nternational {S}ymposium on {C}ommunications and
  {I}nformation {T}echnology}, volume~1, pages 601--604, 2004.

\bibitem{Austin}
T.~L. Austin.
\newblock The enumeration of point labelled chromatic graphs and trees.
\newblock {\em Canadian Journal of Mathematics}, 12:535--545, 1960.

\bibitem{Borchardt}
C.~W. Borchardt.
\newblock {\"U}ber eine {I}nterpolationsformel f\"ur eine {A}rt {S}ymmetrischer
  {F}unctionen und \"uber {D}eren {A}nwendung.
\newblock {\em Berlin, Abhandl}, pages 1--20, 1860.

\bibitem{cayley}
A.~Cayley.
\newblock A theorem on trees.
\newblock {\em Quart. J. Pure Appl. Math.}, 23:376--378, 1889.

\bibitem{Chestnut-Fishkind}
S.~R. Chestnut and D.~E. Fishkind.
\newblock Counting spanning trees of threshold graphs.
\newblock arXiv:1208.4125, 2013.

\bibitem{chvatal_hammer_1977}
V.~Chv\'atal and P.~L. Hammer.
\newblock Aggregation of inequalities in integer programming.
\newblock In P.~Hammer, E.~Johnson, B.~Korte, and G.~Nemhauser, editors, {\em
  {S}tudies in {I}nteger {P}rogramming}, volume~1 of {\em Annals of Discrete
  Mathematics}, pages 145--162. Elsevier, 1977.

\bibitem{E-VW}
R.~Ehrenborg and S.~van Willigenburg.
\newblock Enumerative properties of {F}errers graphs.
\newblock {\em Discrete Comput. Geom.}, 32(4):481--492, 2004.

\bibitem{Hammer-Kelmans}
P.~L. Hammer and A.~K. Kelmans.
\newblock Laplacian spectra and spanning trees of threshold graphs.
\newblock {\em Discrete Appl. Math.}, 65(1-3):255--273, 1996.

\bibitem{Hartsfield-Werth}
N.~Hartsfield and J.~S. Werth.
\newblock Spanning trees of the complete bipartite graph.
\newblock In {\em Topics in {C}ombinatorics and {G}raph {T}heory
  ({O}berwolfach)}, pages 339--346. Physica, Heidelberg, 1990.

\bibitem{HKV}
L.-J. Hung, T.~Kloks, and F.~S. Villaamil.
\newblock Black-and-white threshold graphs.
\newblock In {\em Proceedings of the {S}eventeenth {C}omputing: {T}he
  {A}ustralasian {T}heory {S}ymposium}, volume 119 of {\em {CATS}}, pages
  121--130, 2011.

\bibitem{kirchhoff}
G.~Kirchhoff.
\newblock Uber die auflosung der gleichungen, auf welche man bei der
  untersuchung der linearen verteilung galvanischer strome gefuhrt wird.
\newblock {\em Annalen der Physik}, 148(12):497--508, 1847.

\bibitem{klee_stamps_weighted}
S.~Klee and M.~T. Stamps.
\newblock Linear algebraic techniques for weighted spanning tree enumeration.
\newblock {\em Linear Algebra Appl.}, 582:391--402, 2019.

\bibitem{klee_stamps}
S.~Klee and M.~T. Stamps.
\newblock Linear algebraic techniques for spanning tree enumeration.
\newblock {\em Amer. Math. Monthly}, 127:297--307, 2020.

\bibitem{Lewis}
R.~P. Lewis.
\newblock The number of spanning trees of a complete multipartite graph.
\newblock {\em Discrete Math.}, 197/198:537--541, 1999.

\bibitem{liu}
J.~Liu and H.~Zhou.
\newblock Maximum induced matchings in graphs.
\newblock {\em Adv. Math.}, 170:277--281, 1997.

\bibitem{ma_18}
F.~Ma and B.~Yao.
\newblock An iteration method for computing the total number of spanning trees
  and its applications in graph theory.
\newblock {\em Theoret. Comput. Sci.}, 708:46--57, 2018.

\bibitem{Mahadev-Peled}
N.~V.~R. Mahadev and U.~N. Peled.
\newblock {\em Threshold graphs and related topics}, volume~56 of {\em Annals
  of Discrete Mathematics}.
\newblock North-Holland Publishing Co., Amsterdam, 1995.

\bibitem{Martin-Reiner}
J.~L. Martin and V.~Reiner.
\newblock Factorization of some weighted spanning tree enumerators.
\newblock {\em J. Combin. Theory Ser. A}, 104(2):287--300, 2003.

\bibitem{Merris}
R.~Merris.
\newblock Degree maximal graphs are {L}aplacian integral.
\newblock {\em Linear Algebra Appl.}, 199:381--389, 1994.

\bibitem{Moon}
J.~W. Moon.
\newblock {\em Counting labelled trees}, volume~1 of {\em Canadian Mathematical
  Monographs}.
\newblock Canadian Mathematical Congress, Montreal, QC, 1970.

\bibitem{Onodera}
R.~Onodera.
\newblock On the number of trees in a complete {$n$}-partite graph.
\newblock {\em Matrix Tensor Quart.}, 23:142--146, 1972/73.

\bibitem{Scoins}
H.~I. Scoins.
\newblock The number of trees with nodes of alternate parity.
\newblock {\em Proc. Cambridge Philos. Soc.}, 58:12--16, 1962.

\bibitem{EC1}
R.~P. Stanley.
\newblock {\em Enumerative combinatorics. {V}olume 1}, volume~49 of {\em
  Cambridge Studies in Advanced Mathematics}.
\newblock Cambridge University Press, Cambridge, 2nd edition, 2012.

\bibitem{Temperley}
H.~N.~V. Temperley.
\newblock On the mutual cancellation of cluster integrals in {M}ayer's fugacity
  series.
\newblock {\em Proc. Phys. Soc.}, 83:3--16, 1964.

\end{thebibliography}

\end{document}